\documentclass[11pt,a4paper]{amsart}
\usepackage[headings]{fullpage}
\usepackage{amssymb,amscd,hyperref}
\usepackage{amsthm}
\usepackage[alphabetic, nobysame]{amsrefs}
\usepackage[all]{xy}
\usepackage{color}
\usepackage{xcolor}
\usepackage{tikz-cd} 
\numberwithin{equation}{section}
\newtheorem{thm}{Theorem}[section]
\newtheorem{lem}[thm]{Lemma}
\newtheorem{prop}[thm]{Proposition}

\theoremstyle{definition}
\newtheorem{defn}[thm]{Definition}

\theoremstyle{remark}
\newtheorem{rem}[thm]{Remark}
\newtheorem*{acknowledgements}{Acknowledgements}

\def\Z{\mathbb{Z}}

\def\F{\mathbb{F}}

\def\THH{\mathsf{THH}}
\def\TAQ{\mathsf{TAQ}}
\def\cone{\text{cone}}
\def\Tor{\mathsf{Tor}}
\def\tmf{\mathsf{tmf}}
\def\Tmf{\mathsf{Tmf}}
\def\TMF{\mathsf{TMF}}

\def\ra{\rightarrow}
\def\leq{\leqslant}
\def\geq{\geqslant}

\def\J{\mathcal{J}}

\def\HH{\mathsf{HH}}
\def\Hom{\mathsf{Hom}}

\def\ie{\emph{i.e.}}

\def\id{\mathrm{id}}

\begin{document}
\title[Detecting and describing ramification]{%
Detecting and describing ramification for structured ring spectra}
\author{Eva H\"oning}
\author{Birgit Richter}
\date{\today}

\begin{abstract}
Ramification for commutative ring spectra can be detected by relative
topological Hochschild homology and by topological
Andr\'e-Quillen homology. 
In the classical algebraic context it is important to
distinguish between tame and wild 
ramification. Noether's theorem characterizes tame ramification in
terms of a normal basis and tame ramification can also be detected via
the surjectivity of the trace map. We transfer the latter fact to ring
spectra 
and use the Tate cohomology spectrum to detect wild ramification in
the context of commutative ring spectra. We study ramification in 
examples in the context of topological K-theory and topological modular forms.

\end{abstract}
\maketitle

\section{Introduction}
Classically, ramification is studied in the setting of extensions of
rings of integers in number fields. If $K \subset L$ is an
extension of number fields and if $\mathcal{O}_K \ra \mathcal{O}_L$ is
the corresponding extension of rings of integers, then a prime ideal
$\mathfrak{p} \subset \mathcal{O}_K$ ramifies in $L$, if
$\mathfrak{p}\mathcal{O}_L = \mathfrak{p}_1^{e_1} \cdot \ldots \cdot
\mathfrak{p}_s^{e_s}$ in $\mathcal{O}_L$ and $e_i > 1$ for at least
one $1 \leq i \leq s$. The ramification is \emph{tame} when  the
ramification indices $e_i$ are all relatively prime to the residue
characteristic  of $\mathfrak {p}$ and it is \emph{wild} otherwise. Auslander
and Buchsbaum \cite{ab59} considered ramification in the setting of general
noetherian rings. If $K \subset L$ is a finite $G$-Galois extension, then
$\mathcal{O}_K \ra
\mathcal{O}_L$ is unramified, if and only if $\mathcal{O}_K = \mathcal{O}_L^G\ra
\mathcal{O}_L$ is a Galois extension of commutative rings and this in
turn says that $\mathcal{O}_L \otimes_{\mathcal{O}_K}\mathcal{O}_L
\cong \prod_G \mathcal{O}_L$ (see \cite[Remark 1.5 (d)]{chr}, \cite{ab59} or \cite[Example
2.3.3]{rognes} for more details).

Our main interest is to investigate notions of
ramified extensions of ring spectra and to study examples.

Rognes \cite[Definition 4.1.3]{rognes} introduces 
$G$-Galois extensions of 
ring spectra. A map $A \ra B$ of commutative ring spectra is a
$G$-Galois extension for a finite group $G$, if certain cofibrancy
conditions are satisfied, if $G$ acts on $B$ from the left through
commutative $A$-algebra maps and if the following two conditions are
satisfied:

\begin{enumerate}
\item
  The map from $A$ to the homotopy fixed points of $B$ with respect to
  the $G$-action, $i \colon A \ra B^{hG}$ is a weak equivalence.
\item 
  The map
  \begin{equation} \label{eq:unramified}
    h \colon B \wedge_A B \ra \prod_G B
  \end{equation}
  is a weak equivalence.
\end{enumerate}
Here, $h$ is right
adjoint to the composite map
\[ \xymatrix@1{B \wedge_A B \wedge G_+ \ar[r] & B \wedge_A B \ar[r] &
    B}, \]  induced by the $G$-action $B \wedge G_+ \cong G_+ \wedge B \ra
B$ on $B$ and the multiplication on $B$.

Condition (1) is the fixed point condition familiar from ordinary
Galois theory. Condition (2) is needed to ensure that the map $A \ra
B$ is \emph{unramified}. Among other things, it implies for instance
that the $A$-endomorphisms of $B$ correspond to the group elements in
the sense that 
\[ j \colon B \wedge G_+ \ra F_A(B,B),  \]
is a weak equivalence, where $j$ is right adjoint to the composite map
\[ (B \wedge G_+) \wedge_A B \ra B \wedge_A B \ra B, \]
which is again induced by the $G$-action and the multiplication on $B$.

If $A$ is the Eilenberg-MacLane spectrum $H\mathcal{O}_K$ and $B =
H\mathcal{O}_L$ for a $G$-Galois extension $K \subset L$, then
$H\mathcal{O}_K \ra H\mathcal{O}_L$ is a $G$-Galois extension of ring spectra
if and only if
$\mathcal{O}_K \ra \mathcal{O}_L$ is a $G$-Galois extension of commutative
rings. 

For certain Galois extensions Ausoni and Rognes \cite{ar} conjecture a
version of Galois descent for algebraic K-theory. A descent result that covers
many of the conjectured cases is established in \cite{cmnn}. In some
cases,  descent can be established even in the presence of ramification.
Ausoni \cite[Theorem 10.2]{AuTHHku} shows for instance that the canonical map $K(\ell_p) \ra
K(ku_p)^{hC_{p-1}}$ is an equivalence after $p$-completion despite the fact
that the inclusion of the $p$-completed connective Adams summand, $\ell_p$, into $p$-completed
topological connective K-theory, $ku_p$, should be viewed as a tamely ramified
extension of commutative ring spectra. In other cases that are not Galois
extensions, for instance in cases, that we will identify as wildly ramified,
one can consider a  modified version of descent \cite[\S 5.4]{cmnn}.

How can we detect ramification? 
The unramified condition from \eqref{eq:unramified} ensures for
instance that $A \ra B$ is 
separable \cite[Definition 9.1.1]{rognes} and this in turn implies that the
canonical map from $B$ to the relative topological Hochschild
homology, $\THH^A(B)$, is an equivalence and that the spectrum of topological
Andr\'e-Quillen homology $\TAQ^A(B)$ \cite{b99}  is trivial. So if we know for 
a map of commutative ring spectra $A \ra B$ that $B \ra
\THH^A(B)$ is not a weak equivalence or that $\pi_*\TAQ^A(B)
\neq 0$, then this is an indicator for ramification. We will study
examples of non-vanishing $\TAQ$  in 
\ref{subsec:kaehler} and study relative topological Hochschild homology in
examples related to level-$2$-structures on elliptic curves in
\ref{subsec:thh}. 

An interesting class of examples arises as connective covers of
$G$-Galois extensions. Akhil Mathew shows in \cite[Theorem 6.17]{m16}
that connective Galois extensions are algebraically \'etale: the
induced map on homotopy groups is \'etale in a graded sense. So, in
particular, connective covers of 
Galois extensions are rarely Galois extensions, because several known
examples of Galois extensions such as $KO \ra KU$, $L_p \ra KU_p$ and
examples of Galois extensions in the context of topological modular
forms are far from behaving nicely on the level of homotopy groups.

We use relative topological Hochschild homology and topological
Andr\'e-Quillen homology  in order to detect ramification in the cases
$ko \ra ku$, $\ell \ra ku_{(p)}$, $\tmf_0(3)_{(2)} \ra
\tmf_1(3)_{(2)}$, $\tmf_{(3)} \ra \tmf_0(2)_{(3)}$, $\Tmf_{(3)} \ra 
\Tmf_0(2)_{(3)}$ and $\tmf_0(2)_{(3)} \ra \tmf(2)_{(3)}$. We also study a
version of the discriminant map in the context of structured ring spectra and
apply it to the examples $\ell \ra ku_{(p)}$ and $ko \ra ku$ in
\ref{subsec:discriminant}.

For certain finite extensions of discrete valuation rings tame ramification is
equivalent to being log-\'etale (see for instance \cite[Example
4.32]{ro-log}).   
It is known by work of Sagave \cite{sa}, that $\ell \ra ku_{(p)}$ is log-\'etale
if one considers the log structures generated by $v_1 \in
\pi_{2p-2}\ell$ and $u \in \pi_2(ku)$. We show that $ko \ra ku$ is
\emph{not} log-\'etale if one considers the log structures generated
by the Bott elements $\omega \in \pi_8(ko)$ and $u \in \pi_2(ku)$.


Emmy Noether shows \cite[\S 2]{noe} that tame ramification is equivalent to
the existence of a normal basis. Tame ramification can also be
detected by the surjectivity of the trace map 
\cite[Theorem 2, Chapter 1, \S 5]{cf}. This in turn
yields a vanishing of Tate cohomology. 

Using Tate cohomology as a possible criterion for wild ramification
is for instance suggested by Rognes in \cite{r14}. Rognes also shows a version
of Noether's theorem in \cite[Theorem 5.2.5]{rognes-s}: If a spectrum with a
$G$-action $X$ is in the thick subcategory generated by spectra of the
form $G_+ \wedge W$, then $X^{tG} \simeq *$, so in particular, if $B$
has a normal basis, $B \simeq G_+ \wedge A$, then $B^{tG} \simeq
*$. 

We use the Tate spectrum in order to propose a definition of tame and
wild ramification of maps of ring spectra and study examples in the context
of topological K-theory, topological modular forms and cochains on classifying
spaces with coefficients in Morava E-theory aka Lubin-Tate spectra. 

Several of our examples use topological modular forms with level structures. 
The spectrum of topological modular forms, $\TMF$, 
arises as the global sections of 
a structure sheaf of $E_\infty$-ring spectra on the moduli stack of elliptic 
curves, $\mathcal{M}_{\text{ell}}$. A variant of it, $\Tmf$, lives on a 
compactified version, $\overline{\mathcal{M}}_{\text{ell}}$. Its connective 
version is denoted by $\tmf$. There are other variants corresponding to level 
structures on elliptic curves. Recall that a $\Gamma(n)$-structure 
(or level $n$-structure for short) carries the datum of a chosen isomorphism 
between the $n$-torsion points of an elliptic curve and the group 
$(\Z/n\Z)^2$. A $\Gamma_1(n)$-structure corresponds to the choice of a 
point of exact order $n$ whereas a $\Gamma_0(n)$-structure comes from the
choice 
of a subgroup of order $n$ of the $n$-torsion points. See \cite[Chapter 3]{km} 
for the precise definitions and for background. These level structures give
rise to a tower of moduli problems (see \cite[p.~200]{km})

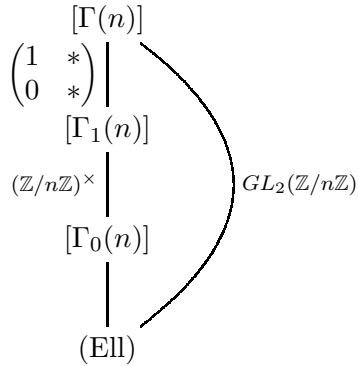
\[
  \xymatrix{
    [\Gamma(n)] \\
    [\Gamma_1(n)] 
    \ar@{-}[u]^{\begin{pmatrix}1 & * \\ 0 & *\end{pmatrix}}\\
    [\Gamma_0(n)]\ar@{-}[u]^{(\Z/n\Z)^\times}\\
    (\text{Ell}) \ar@{-}[u] \ar@{-}@/_10ex/[uuu]_{GL_2(\Z/n\Z)}
  }
\]
with corresponding spectra $\TMF(n)$, $\TMF_1(n)$ and $\TMF_0(n)$ and their
compactified versions $\Tmf(n)$, $\Tmf_1(n)$ and $\Tmf_0(n)$
\cite[Theorem 6.1]{hl}.

In \cite{mm15} Mathew and 
Meier prove that the maps $\Tmf[\frac{1}{n}] \ra \Tmf(n)$ are
\emph{not} Galois extensions but they 
satisfy Tate vanishing, which might be seen as an indication of tame
ramification. In contrast, we will show that $\tmf(n)^{tGL_2(\Z/n\Z)}$
is non-trivial if $2$ does not divide $n$ or if $n$ is a power of $2$.

\begin{acknowledgements}
  We thank Lennart Meier for sharing \cite{meier} with us and we thank
  him, John Rognes and Mark Behrens 
  for several helpful comments. Mike Hill helped us a lot with patient 
  explanations about tmf and friends. The first named author was funded by
  the DFG priority program SPP 1786
  \emph{Homotopy Theory and Algebraic Geometry}. The last named
  author would like to thank the Isaac Newton Institute for Mathematical
  Sciences for support and hospitality during the programme 
  \emph{Homotopy Harnessing Higher Structures}  when work on this paper was
  undertaken. This work was supported by EPSRC grant number EP/R014604/1.
\end{acknowledgements}

\section{Detecting ramification}
\subsection{Topological Andr\'e-Quillen homology} \label{subsec:kaehler}

For a map of connective commutative ring spectra $i \colon A \ra B$ we use the
connectivity of the map to determine the bottom homotopy group of $\TAQ^A(B)$ 
\cite{b99}. The non-triviality of $\TAQ^A(B)$ indicates that the map $i$ is ramified. 

\subsubsection*{Algebraic cases} If $\mathcal{O}_K \ra \mathcal{O}_L$ is 
an extension of number rings with corresponding extension of number fields
$K \subset L$, then of course we cannot use a connectivity argument for
understanding $\TAQ$, but here, the algebraic
module of K\"ahler differentials, $\Omega^1_{\mathcal{O}_L|\mathcal{O}_K}$, is
isomorphic to the first Hochschild homology group
$\HH_1^{\mathcal{O}_K}(\mathcal{O}_L)$ which in turn is
$\pi_0\TAQ^{H\mathcal{O}_K}(H\mathcal{O}_L)$ \cite[Theorem
2.4]{bari}. The ramification type 
of $\mathcal{O}_K \ra \mathcal{O}_L$ can be read off the \emph{different},
\ie, the annihilator of the module of K\"ahler differentials. 

\subsubsection*{The connective Adams summand} Let $\ell$ denote the
Adams summand of connective $p$-localized topological complex K-theory,
$ku_{(p)}$. Here, $p$ is an odd prime.

The inclusion $i \colon \ell \ra ku_{(p)}$ induces an isomorphism on $\pi_0$
and $\pi_1$. Thus, by the Hurewicz theorem for topological Andr\'e-Quillen
homology \cite[Lemma 8.2]{b99}, \cite[Lemma 1.2]{BGR08} we get that
$\pi_2\TAQ^{\ell}(ku_{(p)})$ is the bottom homotopy group and is isomorphic to
the second homotopy group of the
cone of $i$, $\cone(i)$ and this in turn can be determined by the long exact
sequence
\[ \ldots \ra \pi_2(\ell) = 0 \ra \pi_2(ku_{(p)}) \ra \pi_2(\cone(i)) \ra
  \pi_1(\ell) = 0 \ra \ldots \]
hence $\pi_2 \TAQ^{\ell}(ku_{(p)}) \cong \Z_{(p)}$. 

We know from \cite{dlr} that $\ell \ra ku_{(p)}$ shows features of a 
tamely ramified extension of number rings and Sagave shows
\cite[Theorem 6.1]{sa} that 
$\ell \ra ku_{(p)}$ is log-\'etale. 

\subsubsection*{Real and complex connective topological K-theory} The
complexification map $c \colon ko \ra ku$ induces an isomorphism on $\pi_0$
and an epimorphism on $\pi_1$, so it is a $1$-equivalence. Hence again 
$\pi_2\cone(c) \cong \pi_2(\TAQ^{ko}(ku))$ is the bottom homotopy group, but
here we obtain an extension
\[ 0 \ra \pi_2ku = \Z \ra \pi_2\cone(c) \ra \pi_1(ko)=\Z/2\Z \ra 0. \]

In order to understand $\pi_2\cone(c)$ we consider the cofiber sequence 
\[ \xymatrix@1{
\Sigma KO \ar[r]^\eta & KO \ar[r]^c & KU \ar[r]^\delta & \Sigma^2 KO
  }\]
and the commutative diagram on homotopy groups
\[ \xymatrix{
\Z/2\Z = \pi_2ko \ar[r]^(0.6){\pi_2(c)}_(0.6)0
\ar[d]^{\pi_2(\tau_{ko})}_\cong & \pi_2ku \ar[r]
\ar[d]^{\pi_2(\tau_{ku})}_\cong & \pi_2 \cone(c) \ar[r] \ar[d]^g &
\pi_1ko \ar[r]^(0.4){\pi_1(c)} \ar[d]^{\pi_1(\tau_{ko})}_\cong & 0 =
\pi_1ku  \ar[d]^{\pi_1(\tau_{ku})}_\cong \\   
\Z/2\Z = \pi_2KO \ar[r]^(0.6){\pi_2(c)}_(0.6)0& \pi_2KU \ar[r]^{\pi_2(\delta)}&
\pi_2\Sigma^2 KO \ar[r]^{\pi_2(\Sigma \eta)} & \pi_1KO
\ar[r]^(0.4){\pi_1(c)}& 0 = \pi_1KU.  }\]
Here $\tau_e \colon e \ra E$ denotes the map from the connective cover
$e$ of $E$ to $E$. The middle vertical map $g$ is the map induced by the
cofiber sequences. By the five lemma, $g$ is an isomorphism hence
\[  \pi_2 \cone(c) \cong \pi_2\Sigma^2 KO \cong \Z. \]
So this group is also torsion free. We will later see that $ko \ra ku$
is not log-\'etale and we will see some other indicators for wild
ramification, but the bottom homotopy group of
$\TAQ^{ko}(ku)$ does not detect that. 
\subsubsection*{Connective topological modular forms with level structure
  (case $n=3$) }
We consider $\tmf_1(3)$. Its homotopy groups are $\pi_*(\tmf_1(3)) \cong
\Z[\frac{1}{3}][a_1, a_3]$ with $|a_i| = 2i$. See \cite{hl} for some
background. There is a $C_2$-action on $\tmf_1(3)$ coming from the permutation
of elements of exact order three and one denotes by $\tmf_0(3)$ the connective
cover of the homotopy fixed points, $\tmf_1(3)^{hC_2}$. There is a homotopy
fixed point spectral sequence that was studied in detail in \cite{mr09} for
the periodic versions. In \cite[p.~407]{hl} it is explained how to adapt this
calculation to the connective variants: The terms in the spectral sequence
with $s > t-s \geq 0$ can be ignored. The $C_2$-action on the $a_i$'s is
given by the sign-action, so if $\tau$ generates $C_2$, then $\tau(a_i^n)
=(-1)^na_i^n$. 

This implies
that only $H^0(C_2; \pi_0(\tmf_1(3))) \cong \Z[\frac{1}{3}]$ survives to
$\pi_0(\tmf_0(3))$. For $\pi_1$ we get a contribution from $H^1(C_2;
\pi_2(\tmf_1(3)))$, giving a $\Z/2\Z$
generated by the class of $a_1$ (this detects an $\eta$). For
$\pi_2(\tmf_0(3))$ the class of $a_1^2$ generates a copy of $\Z/2\Z$. 

Hence the map $j \colon \tmf_0(3)_{(2)} \ra \tmf_1(3)_{(2)}$ is $1$-connected, so
$\pi_2\TAQ^{\tmf_0(3)_{(2)}}(\tmf_1(3)_{(2)})$ is the bottom homotopy group
and is isomorphic to $\pi_2(\cone(j))$ which sits in an extension.
We can use the commutative diagram of commutative ring spectra from
\cite[Theorem 63]{hl}
\[ \xymatrix{
    \tmf_0(3) \ar[r] \ar[d] & \tmf_1(3) \ar[d] \\
    ko[\frac{1}{3}] \ar[r] & ku[\frac{1}{3}]
  } \]
in order to determine to $\pi_2(\cone(j))$. By
\cite[Theorem 1.2]{ln} there is
a cofiber sequence of $\tmf_1(3)_{(2)}$-modules
\[ \xymatrix@1{\Sigma^6\tmf_1(3)_{(2)} \ar[r]^{v_2} & \tmf_1(3)_{(2)} \ar[r] &
    ku_{(2)}} \] and hence
  $\pi_2(\tmf_1(3)_{(2)}) \cong \pi_2(ku_{(2)})$. 

The diagram 
\[ \xymatrix{
\pi_2\tmf_0(3)_{(2)} \ar[r]^0 \ar[d]^\cong & \pi_2\tmf_1(3)_{(2)} \ar[d]^\cong \ar[r] & 
\pi_2\cone(j) \ar[d] \ar[r] & \pi_1\tmf_0(3)_{(2)} \ar[d]^\cong \ar[r] &
\pi_1\tmf_1(3)_{(2)} = 0 \ar[d]^\cong\\
\pi_2ko_{(2)} \ar[r]^0 & \pi_2ku_{(2)}  \ar[r] &
\pi_2\cone(c) \ar[r] & \pi_1ko_{(2)}\ar[r]^0 &
\pi_1ku_{(2)}= 0.  
  } \]
commutes and the $5$-lemma implies that $\pi_2(\cone(j)) \cong
\Z_{(2)}$.  

\subsubsection*{Connective topological modular forms with level structure
  (case $n=2$, $p=3$)}
Forgetting a $\Gamma_0(2)$-structure yields a map $f\colon \tmf_{(3)} \ra
\tmf_0(2)_{(3)}$ such that $f$ is a $3$-equivalence. We will recall more details about these spectra at the beginning of \ref{subsec:thh}. Again, we obtain that
the bottom non-trivial homotopy group of the spectrum of topological Andr\'e-Quillen homology is
$\pi_4(\TAQ^{\tmf_{(3)}}(\tmf_0(2)_{(3)})) \cong
\pi_4(\cone(f))$. There is a short exact sequence
\[ 0 = \pi_4\tmf_{(3)} \ra \pi_4\tmf_0(2)_{(3)} = \Z_{(3)} \ra
  \pi_4\cone(f) \ra \pi_3\tmf_{(3)} \cong \Z/3\Z \ra 0\] 
so a priori $\pi_4 \cone(f)$ could be isomorphic to $\Z_{(3)}$ or to
$\Z_{(3)} \oplus \Z/3\Z$.

 There is an equivalence \cite[\S 2.4]{behrens}
\[ \tmf_{(3)} \wedge T \simeq \tmf_0(2)_{(3)}\]
where $T = S^0 \cup_{\alpha_1} e^4 \cup_{\alpha_1} e^8$ with
$\alpha_1$ denoting the generator of $\pi_3^s$ at $3$. Thus $T$ is
part of a cofiber sequence
\[ S^0 \ra T \ra \Sigma^4 \cone(\alpha_1) \]
and we obtain a cofiber sequence
\[ \tmf_{(3)} = \tmf_{(3)} \wedge S^0 \ra \tmf_{(3)} \wedge T \simeq
  \tmf_0(2)_{(3)} \ra   \tmf_{(3)} \wedge \Sigma^4 \cone(\alpha_1)\]
and thus  
\[ \pi_4(\cone(f)) \cong \pi_4 (\tmf_{(3)} \wedge \Sigma^4
  \cone(\alpha_1)) \cong \pi_0 (\tmf_{(3)} \wedge  \cone(\alpha_1)). \] 
But as we have a short exact sequence
\[ 0 = \pi_0(\Sigma^3 \tmf_{(3)})  \ra \pi_0 (\tmf_{(3)}) \cong \Z_{(3)}
  \ra \pi_0 (\tmf_{(3)} \wedge  \cone(\alpha_1)) \ra 0 \]
we obtain
\[ \pi_4(\TAQ^{\tmf_{(3)}}(\tmf_0(2)_{(3)}))  \cong \Z_{(3)}. \] 

\subsection{Relative topological Hochschild homology} \label{subsec:thh}

In \cite{dlr} (see also \cite{dlrerratum} for a correction) we show that the 
relative topological Hochschild homology spectra $\THH^{\ell}(ku_{(p)})$ and
$\THH^{ko}(ku)$ have highly non-trivial homotopy groups. Here, we extend these
results to the relative $\THH$-spectra of $\tmf_{(3)} \to \tmf_0(2)_{(3)}$, $\Tmf_{(3)} \to \Tmf_0(2)_{(3)}$ and $\tmf_0(2)_{(3)} \to \tmf(2)_{(3)}$.   
For formulas concerning the coefficients of elliptic curves we refer to
\cite{deligne}. 

Recall that we  have $\tmf_0(2)_{(3)} \simeq \tau_{\geq 0}\tmf(2)_{(3)}^{hC_2}$.  By \cite[\S 7]{st} 
we know that $\pi_*\tmf(2)_{(3)} \cong \Z_{(3)}[\lambda_1, \lambda_2]$
with $|\lambda_i| = 4$ and with $C_2$-action given by $\lambda_1 \mapsto \lambda_2$  and $\lambda_2 \mapsto \lambda_1$  \cite[Lemma 7.3]{st}. Since $|C_2|$ is invertible in
$\pi_*\tmf(2)_{(3)}$ the $E^2$-page of the homotopy fixed point spectral sequence is
given by 
\[ H^*(C_2, \pi_*\tmf(2)_{(3)}) = H^0(C_2, \pi_*\tmf(2)_{(3)}) = \pi_*(\tmf(2)_{(3)})^{C_2}.
\] 
Thus, we have
\[ \pi_*\tmf_0(2)_{(3)} = \Z_{(3)}[\lambda_1 + \lambda_2, \lambda_1\lambda_2] =
  \Z_{(3)}[a_2, a_4] \] 
with $a_2 = -(\lambda_1 + \lambda_2)$ and $a_4 = \lambda_1 \lambda_2$. 
Recall the following facts about the homotopy of $\tmf_{(3)}$ (see for instance \cite[p.~192]{tmfbook}):  We have 
\[
\pi_*\tmf_{(3)} = \begin{cases}
                           \Z_{(3)}\{1\},  &  * = 0; \\
                           \Z/3\Z\{\alpha_1\}  &   * = 3; \\
                           \Z_{(3)}\{c_4\}, & * = 8; \\
                           \Z_{(3)}\{c_6\}, & * = 12; \\ 
                           0,  & * = 4,5,6,7;
                         \end{cases} \] 
where $\alpha_1$ is the image of $\alpha_1 \in \pi_3(S_{(3)})$ under 
$\pi_3(S_{(3)}) \to \pi_3\tmf_{(3)}$. 
By \cite[Proof of Proposition 10.3]{st} we have that the map
$\pi_*\tmf_{(3)} \to \pi_*\tmf_0(2)_{(3)}$ satisfies 
$c_4 \mapsto  16a_2^2 -48 a_4$ and $c_6 \mapsto -64 a_2^3 + 288 a_2a_4$. (There is a
discrepancy between our sign for $c_6$ and that in \cite{st}.) 

We know from personal communication with Mike Hill that there is  a fiber
sequence of $\tmf_0(2)_{(3)}$-modules 
\[  \begin{tikzcd} 
    \Tmf_0(2)_{(3)} \ar{r} & \tmf_0(2)_{(3)}[a_2^{-1}] \times
    \tmf_0(2)_{(3)}[a_4^{-1}] \ar{r}{f} & \tmf_0(2)_{(3)}[(a_2a_4)^{-1}].  
\end{tikzcd} \]  
See \cite[Proposition 4.24]{hm} for the analogous statement at $p = 2$.  
The kernel of $\pi_*(f)$  has $\mathbb{Z}_{(3)}$-basis 
\[ \{(a_2^n a_4^m, -a_2^n a_4^m) |  n, m \in \mathbb{N} \} \]  
and the cokernel has $\mathbb{Z}_{(3)}$-basis 
\[ \{ \frac{1}{a_2^n a_4^m} \mid n \geq 1, m \geq 1 \}. \] 
We get that  in negative degrees $\pi_*\Tmf_0(2)_{(3)}$ is given by 
\[ \bigoplus_{n,m \geq 1} \mathbb{Z}_{(3)}\{ \frac{1}{a_2^na_4^m}\}, \] 
where $\frac{1}{a_2^n a_4^m}$ has degree  $-4n-8m-1$. The
 $\pi_*\tmf_0(2)_{(3)}$-action is  given by 
\[ a_2 \cdot \frac{1}{a_2^na_4^m} = \begin{cases} 
\frac{1}{a_2^{n-1}a_4^m}, &  \text{ if }  n \geq 2 \\
                                                 0, & \text{ otherwise},
                                               \end{cases} \] 
 and analogously for $a_4$. 
 
 By the gap theorem (see for instance
 \cite{konter}) we have $\pi_*\Tmf_{(3)}  \cong 0$ for $-21 <  * < 0$.

\begin{lem} \label{relsmash1}
\[ \pi_*(\tmf_0(2)_{(3)} \wedge_{\tmf_{(3)}} \tmf_0(2)_{(3)}) \cong
  \Z_{(3)}[a_2,a_4,r]/r^3+a_2r^2 +a_4 r =: \Z_{(3)}[a_2,a_4,r]/I\]
where $r$ has degree $4$ and is mapped to zero under the multiplication map. 
\end{lem} 
\begin{proof}
As above we use that we have  an equivalence of $\tmf_{(3)}$-modules $\tmf_{(3)} \wedge T \simeq \tmf_0(2)_{(3)}$.
Here, $T$ is defined by the cofiber sequences 
\[ \begin{tikzcd} 
S^3_{(3)} \ar{r}{\alpha_1} & S^0_{(3)}  \ar{r} & \cone(\alpha_1) \ar{r} & S^4_{(3)} 
\end{tikzcd}\]
and 
\[ \begin{tikzcd} 
S^7_{(3)} \ar{r}{\phi}  &  \cone(\alpha_1) \ar{r} &  T \ar{r} & S^8_{(3)}, 
\end{tikzcd}\]
where $S^7_{(3)} \xrightarrow{\phi} \cone(\alpha_1) \to S^4_{(3)}$ is equal to $\alpha_1$. 
We get an equivalence of  left $\tmf_0(2)_{(3)}$-modules
\[\tmf_0(2)_{(3)} \wedge_{\tmf_{(3)}} \tmf_0(2)_{(3)} \simeq  \tmf_0(2)_{(3)} \wedge_{\tmf_{(3)}} (\tmf_{(3)} \wedge T) \simeq \tmf_0(2)_{(3)} \wedge T. \]
Smashing the above cofiber sequences with $\tmf_0(2)_{(3)}$ 
gives  cofiber sequences of $\tmf_0(2)_{(3)}$-modules
\[ \begin{tikzcd} 
\Sigma^3 \tmf_0(2)_{(3)} \ar{r}{\bar{\alpha}_1} & \tmf_0(2)_{(3)}  \ar{r} & \tmf_0(2)_{(3)} \wedge \cone(\alpha_1)  \ar{r}{\delta} & \Sigma^4 \tmf_0(2)_{(3)} 
\end{tikzcd}\] 
and 
\[\begin{tikzcd} 
\Sigma^7 \tmf_0(2)_{(3)} \ar{r}{\bar{\phi}} & \tmf_0(2)_{(3)} \wedge \cone(\alpha_1) \ar{r} &  \tmf_0(2)_{(3)} \wedge T \ar{r}{\Delta} & \Sigma^8 \tmf_0(2)_{(3)}.
\end{tikzcd}\] 
The map $\bar{\alpha}_1$ is zero in the derived category of $\tmf_0(2)_{(3)}$-modules, because $\pi_*(\tmf_0(2)_{(3)})$  is concentrated in even degrees.  We therefore get  an equivalence of $\tmf_0(2)_{(3)}$-modules 
\[ \tmf_0(2)_{(3)} \wedge \cone(\alpha_1) \simeq \tmf_0(2)_{(3)}  \vee \Sigma^4 \tmf_0(2)_{(3)}.\] 
This implies  that $\tmf_0(2)_{(3)} \wedge \cone(\alpha_1)$ has non-trivial homotopy groups only in even degrees and  therefore that $\bar{\phi}$ is zero in the derived category of $\tmf_0(2)_{(3)}$-modules.  We get an equivalence of $\tmf_0(2)_{(3)}$-modules 
\[ \tmf_0(2)_{(3)} \wedge T  \simeq \tmf_0(2)_{(3)} \vee \Sigma^4\tmf_0(2)_{(3)}  \vee \Sigma^8 \tmf_0(2)_{(3)}.\]
We can assume that  the map  $\tmf_{(3)} \to \tmf_0(2)_{(3)}$ factors in the derived category of $\tmf_{(3)}$-modules as 
\[\begin{tikzcd}
\tmf_{(3)} \ar{r} & \tmf_{(3)} \wedge \cone(\alpha_1) \ar{r} & \tmf_{(3)} \wedge T \ar{r}{\simeq} & \tmf_0(2)_{(3)} 
\end{tikzcd}\]
This implies that the inclusion in the first smash factor 
\[\eta_L\colon  \tmf_0(2)_{(3)} \to \tmf_0(2)_{(3)} \wedge_{\tmf_{(3)}} \tmf_0(2)_{(3)} \]
is given by 
\[\begin{tikzcd}
    \tmf_0(2)_{(3)} \ar{r} & \tmf_0(2)_{(3)} \wedge \cone(\alpha_1) \ar{r} & \tmf_0(2)_{(3)} \wedge T
    \ar{r}{\simeq} & \tmf_0(2)_{(3)} \wedge_{\tmf_{(2)}} \tmf_0(2)_{(3)}.
\end{tikzcd}\] 
We obtain that the map 
\[\begin{tikzcd}
    \tmf_0(2)_{(3)} \wedge \cone(\alpha_1) \ar{r} & \tmf_0(2)_{(3)} \wedge T \simeq \tmf_0(2)_{(3)}
    \wedge_{\tmf_{(3)}} \tmf_0(2)_{(3)} \ar{r} & \tmf_0(2)_{(3)}
\end{tikzcd}\]
is a left inverse  for $\tmf_0(2)_{(3)} \to \tmf_0(2)_{(3)} \wedge
\cone(\alpha_1)$. 
It is also clear that the inclusion in the second smash factor $\eta_R \colon
\tmf_0(2)_{(3)} \to \tmf_0(2)_{(3)} \wedge_{\tmf_{(3)}} \tmf_0(2)_{(3)}$  is given by
\[\begin{tikzcd}
\tmf_0(2)_{(3)} \ar{r}{\simeq}  & \tmf_{(3)} \wedge T \ar{r} &  \tmf_0(2)_{(3)} \wedge T \ar{r}{\simeq}  & \tmf_0(2)_{(3)} \wedge_{\tmf_{(3)}} \tmf_0(2)_{(3)}.
\end{tikzcd}\]
We claim that  
\[ \eta_R(a_2) \in \pi_4(\tmf_0(2)_{(3)} \wedge_{\tmf_{(3)}} \tmf_0(2)_{(3)})
  \cong \pi_4(\tmf_0(2)_{(3)} \wedge T) \cong \pi_4(\tmf_0(2)_{(3)} \wedge \cone(\alpha_1)) \] 
 maps to $3$ times a unit under 
 \[ \delta_4 \colon  \pi_4(\tmf_0(2)_{(3)} \wedge \cone(\alpha_1)) \to \pi_4(\Sigma^4 \tmf_0(2)_{(3)}) \cong \mathbb{Z}_{(3)}.\]
 By commutativity of the diagram
 \[ \begin{tikzcd}
     \pi_4(\tmf_0(2)_{(3)} \wedge T)  & \pi_4(\tmf_0(2)_{(3)} \wedge \cone(\alpha_1))
     \ar{l}[swap]{\cong} \ar{r}{\delta_4} & \pi_4(\Sigma^4 \tmf_0(2)_{(3)})  \\
     \pi_4(\tmf_{(3)} \wedge T) \ar{u}{\eta_R}  & \pi_4(\tmf_{(3)} \wedge \cone(\alpha_1)) \ar{r}
     \ar{l}[swap]{\cong} \ar{u} &  \pi_4(\Sigma^4\tmf_{(3)})  \ar{u} 
\end{tikzcd} \] 
it suffices to show that $a_2 \in \pi_4(\tmf_{(3)} \wedge T)$ maps to $3$ times a unit
under the bottom map. 
This follows by the exact sequence 
\[\begin{tikzcd} 
    \pi_4(\tmf_{(3)}) \ar{r}  \ar{d}{\cong} & \pi_4(\tmf_{(3)} \wedge \cone(\alpha_1))  \ar{r} 
    \ar{d}{\cong} &  \pi_4(\Sigma^4 \tmf_{(3)})  \ar{r} \ar{d}{\cong}  &
    \pi_4(\Sigma \tmf_{(3)}) \ar{d}{\cong}  \\
    0   \ar{r}  & \mathbb{Z}_{(3)}  \ar{r}   & \mathbb{Z}_{(3)}   \ar{r} &
    \mathbb{Z}/{3\mathbb{Z}}\{\alpha_1\}. 
\end{tikzcd} \] 
We define $r$ to be the unique element in $\pi_4(\tmf_0(2)_{(3)} \wedge \cone(\alpha_1)) $ that maps to
that unit under $\delta_4$ and  that is in the kernel of the composition of 
$\pi_4(\tmf_0(2)_{(3)} \wedge \cone(\alpha_1))  \cong  \pi_4(\tmf_0(2)_{(3)} \wedge T)$ and the multiplication map 
\[ \pi_4(\tmf_0(2)_{(3)} \wedge T) \cong \pi_4(\tmf_0(2)_{(3)} \wedge_{\tmf_{(3)}} \tmf_0(2)_{(3)}) \to \pi_4(\tmf_0(2)_{(3)}). \] 
We have that $3r- \eta_R(a_2)$ is in the image of  
$\pi_4(\tmf_0(2)_{(3)}) \to \pi_4(\tmf_0(2)_{(3)} \wedge \cone(\alpha_1))$ and thus can be written as
$3r- \eta_R(a_2)  = n \cdot a_2$ for an $n \in \mathbb{Z}_{(3)}$.  Applying the map 
\[ \pi_4(\tmf_0(2)_{(3)} \wedge \cone(\alpha_1))  \cong \pi_4(\tmf_0(2)_{(3)} \wedge T)  \cong
  \pi_4(\tmf_0(2)_{(3)} \wedge_{\tmf_{(3)}} \tmf_0(2)_{(3)}) \to  \pi_4(\tmf_0(2)_{(3)}) \]
gives $n = -1$.

We claim that $\eta_R(a_4) \in \pi_8(\tmf_0(2)_{(3)} \wedge T)$ maps to $3$ times a unit
under 
\[\Delta_8 \colon  \pi_8(\tmf_0(2)_{(3)} \wedge T) \to \pi_8(\Sigma^8\tmf_0(2)_{(3)}). \] 
As above one sees that it suffices to show that $a_4$ maps to $3$ times a unit under the
map
$\pi_8(\tmf_{(3)} \wedge T) \to \pi_8(\Sigma^8 \tmf_{(3)})$. 
For this we consider the exact sequence 
\[\begin{tikzcd}
\pi_8(\tmf_{(3)} \wedge \cone(\alpha_1)) \ar{r} & \pi_8(\tmf_{(3)} \wedge T) \ar{r} & \pi_8(\Sigma^8\tmf_{(3)}) \ar{r} & \pi_8(\Sigma \tmf_{(3)}  \wedge \cone(\alpha_1)). \\
 & \mathbb{Z}_{(3)}\{a_2^2\} \oplus \mathbb{Z}_{(3)}\{a_4\} \ar{u}{\cong} & \mathbb{Z}_{(3)} \ar{u}{\cong} & 
\end{tikzcd} \]
Using that $\pi_4(\tmf_{(3)}) = 0 = \pi_5(\tmf_{(3)})$ one gets that 
$\pi_8(\tmf_{(3)}) \cong \pi_8(\tmf_{(3)} \wedge \cone(\alpha_1))$, and under this isomorphism the
first map in the exact sequence identifies with 
\[\pi_8(\tmf_{(3)}) \cong \mathbb{Z}_{(3)}\{c_4\} \to \pi_8(\tmf_0(2)_{(3)}) \cong \mathbb{Z}_{(3)}\{a_2^2\} \oplus \mathbb{Z}_{(3)}\{a_4\}, \quad  c_4 \mapsto 16a_2^2- 48a_4.\] 
As $\pi_6(\tmf_{(3)}) = 0 = \pi_7(\tmf_{(3)})$, one gets that 
$\pi_7(\tmf_{(3)} \wedge \cone(\alpha_1)) \cong \pi_7(\Sigma^4 \tmf_{(3)})$, and under this isomorphism
the third map in the exact sequence identifies with
\[ \pi_8(\Sigma^8\tmf_{(3)}) \cong \mathbb{Z}_{(3)} \to \pi_3(\tmf_{(3)}) \cong \mathbb{Z}/{3\mathbb{Z}}\{\alpha_1\}, \quad 1 \mapsto \alpha_1.\] 
One obtains that the second map in the exact sequence maps  $a_4$ to  $3 \cdot m$  and $a_2^2$ to $9 \cdot m$ for a unit $m \in\mathbb{Z}_{(3)}$.

Since the map $\pi_*(\tmf_{(3)}) \to \pi_*(\tmf_0(2)_{(3)})$ maps $c_4 $ to
$16a_2^2 -48 a_4$, we have the equation 
\[ 16 \cdot a_2^2 - 48 \cdot  a_4 = 16 \cdot \eta_R(a_2)^2-
  48 \cdot \eta_R(a_4) \]
in $\pi_*(\tmf_0(2)_{(3)} \wedge_{\tmf_{(3)}} \tmf_0(2)_{(3)})$. 
Replacing $\eta_R(a_2)$ by $3r + a_2$  and using torsion-freeness 
one gets the equation
\[\eta_R(a_4) = a_4 + 3r^2 + 2 a_2 r.\] 
We apply the map $\Delta_8\colon
\pi_8(\tmf_0(2)_{(3)} \wedge T) \to \pi_8(\Sigma^8 \tmf_0(2)_{(3)})$  to this equation
and obtain by torsion-freeness of $\pi_*(\tmf_0(2)_{(3)})$  
\[ \Delta_8(r^2) = m. \]
We thus have an isomorphism of left $\pi_*(\tmf_0(2)_{(3)})$-modules 
\[ \pi_*(\tmf_0(2)_{(3)} \wedge_{\tmf_{(3)}} \tmf_0(2)_{(3)}) \cong
  \pi_*(\tmf_0(2)_{(3)}) \oplus \pi_*(\tmf_0(2)_{(3)})\{r\} \oplus \pi_*(\tmf_0(2)_{(3)}) \{r^2\}.\] 
Since the map $\pi_*(\tmf_{(3)}) \to \pi_*(\tmf_0(2)_{(3)})$ maps $c_6$ to 
$-64 a_2^3 + 288 a_2 a_4$,  
 we have 
\[ -64 \cdot a_2^3 + 288 \cdot a_2\cdot  a_4 = -64 \cdot \eta_R(a_2)^3 + 288 \cdot  \eta_R(a_2)\cdot \eta_R(a_4) \]
in $\pi_*(\tmf_0(2)_{(3)} \wedge_{\tmf_{(3)}} \tmf_0(2)_{(3)})$. 
Replacing $\eta_R(a_2)$ by $3r + a_2$ and $\eta_R(a_4)$ by $a_4 + 3r^2 + 2a_2 r$ and using torsion-freeness one gets 
\[ r^3 + a_2 r^2 + a_4 r = 0.\] 
This implies the  lemma. 
\end{proof}

  \begin{thm}
    The relative $\THH$-spectrum $\THH^{\tmf_{(3)}}(\tmf_0(2)_{(3)})$ is
    not trivial. More precisely, 
  \begin{align*} \THH_*^{\tmf_{(3)}}(\tmf_0(2)_{(3)}) & \cong \Z_{(3)}[a_2,a_4]
                                            \oplus \bigoplus_{i \geq 
                                            0}    \Sigma^{14i +5}
                                            \Z_{(3)}[a_2] \\ 
    & \cong \pi_*\tmf_0(2) \oplus \bigoplus_{i \geq 0}    \Sigma^{14i
      +5} \pi_*\tmf_0(2)/(a_4).  \end{align*}
  \end{thm}
\begin{proof}
We use the Tor spectral sequence
\[ E^2_{*,*} = \Tor_{*,*}^{\pi_*(\tmf_0(2)_{(3)} \wedge_{\tmf_{(3)}}
    \tmf_0(2)_{(3)})}(\pi_*\tmf_0(2)_{(3)},\pi_*\tmf_0(2)_{(3)})  \Rightarrow 
  \pi_*\THH^{\tmf_{(3)}}(\tmf_0(2)_{(3)}) \]
in order to calculate relative topological Hochschild homology. 
For determining
\[ \Tor_{*,*}^{\Z_{(3)}[a_2,a_4,r]/I}(\Z_{(3)}[a_2,a_4],
\Z_{(3)}[a_2,a_4])\]  
  we consider the  free 
  resolution of $\Z_{(3)}[a_2,a_4]$ as a
  $\Z_{(3)}[a_2,a_4,r]/I$-module
  \[ \xymatrix@1{\ldots \ar[r] &
      \Sigma^{12}\Z_{(3)}[a_2,a_4,r]/I
      \ar[rrr]^{r^2+a_2r+a_4}&& & \Sigma^4\Z_{(3)}[a_2,a_4,r]/I  \ar[r]^r &  \Z_{(3)}[a_2,a_4,r]/I.} \] 
 Applying $(-) \otimes_{\Z_{(3)}[a_2,a_4,r]/I} \Z_{(3)}[a_2,a_4]$ yields
 \[ \xymatrix@1{\ldots \ar[r] &
      \Sigma^{12}\Z_{(3)}[a_2,a_4] \ar[r]^{a_4} &
      \Sigma^4\Z_{(3)}[a_2,a_4] \ar[r]^0 &  \Z_{(3)}[a_2,a_4]} \]
and hence we get      
      
  \[ E^2_{n,*} =  \begin{cases}
              \pi_*(\tmf_0(2)_{(3)}),    & n = 0; \\
              \Sigma^{4 +12k} \mathbb{Z}_{(3)}[a_2],  &n = 2k+1, k \geq 0;  \\
              0,   & \text{otherwise}.
             \end{cases}
           \]
We note that all non-trivial classes in positive filtration degree have an odd total 
degree. Since the edge morphism $\pi_*(\tmf_0(2)_{(3)}) \to
\THH^{\tmf_{(3)}}_*(\tmf_0(2)_{(3)})$ is the unit, the classes in filtration degree zero
cannot be hit by a differential and  the spectral seqence collapses at the $E^2$-page. 
Since $E^2_{n,m} = E^\infty_{n,m}$ is a free $\mathbb{Z}_{(3)}$-module for all $n,m$,
there are no additive extensions. 

\end{proof}


As for the connective covers we have an equivalence of $\Tmf_{(3)}$-modules
$\Tmf_{(3)} \wedge T \simeq \Tmf_0(2)$ \cite[\S 4.6]{mathewh}  
such that the map $\Tmf_{(3)} \to \Tmf_0(2)_{(3)}$ factors in the derived category of
$\Tmf_{(3)}$-modules as  
\[ \Tmf_{(3)} \to \Tmf_{(3)} \wedge \cone(\alpha_1) \to \Tmf_{(3)} \wedge T \simeq \Tmf_0(2)_{(3)}. \]  
Using the gap theorem, one can argue 
analogously to the proof of  Lemma \ref{relsmash1} to show that 
\[ \pi_*(\Tmf_0(2)_{(3)} \wedge_{\Tmf_{(3)}} \Tmf_0(2)_{(3)})
  \cong \pi_*\Tmf_0(2)_{(3)}[r]/{(r^3 + a_2r^2 + a_4r)}.\]

\begin{thm}
There is an additive isomorphism 
\[ \THH^{\Tmf_{(3)}}(\Tmf_0(2)_{(3)}) \cong \pi_*\Tmf_0(2)_{(3)} \oplus
  \bigoplus_{i \geq 0}\Sigma^{14i + 5} \mathbb{Z}_{(3)}[a_2]
  \oplus \bigoplus_{i \geq 1}\Sigma^{14i} \mathbb{Z}_{(3)}\{\frac{1}{a_2^i a_4}\}. \]  
\end{thm}
\begin{proof}
  As above we have the following free resolution of $\pi_*(\Tmf_0(2)_{(3)})$ as
  a module over 
\[ C_* =   \pi_*(\Tmf_0(2)_{(3)} \wedge_{\Tmf_{(3)}} \Tmf_0(2)_{(3)}): \] 
\[ \begin{tikzcd}[column sep = large]
\dots  \ar{r}{r} & \Sigma^{12}C_* \ar{rr}{r^2 + a_2 r + a_4} & &  \Sigma^4 C_* \ar{r}{r} & C_* \ar{r} & \pi_*\Tmf_0(2)_{(3)}   \ar{r} & 0. 
\end{tikzcd}  \] 
We get that the $E^2$-page of the Tor spectral sequence
\[ E^2_{*,*} = \Tor^{C_*}_{*,*}(\pi_*\Tmf_0(2)_{(3)}, \pi_*\Tmf_0(2)_{(3)}) \Longrightarrow \pi_*\THH^{\Tmf_{(3)}}(\Tmf_0(2)_{(3)}) \] 
is given by 
\begin{align*}
 E^2_{n, *} & =    \begin{cases} 
                        \pi_*\Tmf_0(2)_{(3)},  &  n = 0; \\
                             \Sigma^{4+12k}   \pi_*\Tmf_0(2)_{(3)}/{a_4} ,  & n = 2k  + 1, k \geq 0;  \\
                         \ker\bigl( \Sigma^{12k} \pi_*\Tmf_0(2)_{(3)} \xrightarrow{\cdot a_4} \Sigma^{12(k-1) +4} \pi_*\Tmf_0(2)_{(3)}\bigr), & n = 2k, k  > 0 
                      \end{cases}  \\ 
 & =  \begin{cases} 
 \pi_*\Tmf_0(2)_{(3)},  &  n = 0;   \\
 \Sigma^{4+12k} \mathbb{Z}_{(3)}[a_2], & n = 2k+1, k \geq 0; \\
 \Sigma^{12k} \bigoplus_{n \geq 1} \mathbb{Z}_{(3)}\{\frac{1}{a_2^n a_4}\}, & n = 2k, k > 0. 
 \end{cases}   
  \end{align*}
  Since all non-trivial classes in positive filtration   have an odd total degree,
  the spectral sequence collapses at the $E^2$-page. 
There are no additive extensions, because the $E^{\infty} = E^2$-page is a free 
$\mathbb{Z}_{(3)}$-module.  
\end{proof}
\begin{thm}
We have an additive isomorphism 
\[ \pi_*\THH^{\tmf_0(2)_{(3)}}(\tmf(2)_{(3)}) \cong
  \mathbb{Z}_{(3)}[\lambda_1, \lambda_2]  
\oplus  \bigoplus_{i \geq 0} \Sigma^{10i + 5}\mathbb{Z}_{(3)}[\lambda_1] .\] 
\end{thm}
\begin{proof}
The map $\pi_*\tmf_0(2)_{(3)} \to \pi_*\tmf(2)_{(3)}$ is given by
$\mathbb{Z}_{(3)}[\lambda_1 + \lambda_2, \lambda_1\lambda_2] \to
\mathbb{Z}_{(3)}[\lambda_1, \lambda_2]$.  One easily sees that 
\[ \mathbb{Z}_{(3)}[\lambda_1, \lambda_2] \cong
  \mathbb{Z}_{(3)}[\lambda_1 + \lambda_1, \lambda_1\lambda_2]
  \oplus \mathbb{Z}_{(3)}[\lambda_1 + \lambda_2, \lambda_1\lambda_2]\lambda_1,\] 
so $\pi_*\tmf(2)_{(3)}$ is a free $\pi_*\tmf_0(2)_{(3)}$-module. We get 
\begin{align*}
  \pi_*(\tmf(2)_{(3)} \wedge_{\tmf_0(2)_{(3)}} \tmf(2)_{(3)}) & = 
\pi_*\tmf(2)_{(3)} \otimes_{\pi_*\tmf_0(2)_{(3)}} \pi_*\tmf(2)_{(3)} \\
 & =   \mathbb{Z}_{(3)}[\lambda_1, \lambda_2, a]/{a^2  + \lambda_1 a - \lambda_2 a},
 \end{align*}
 where $a = \eta_R(\lambda_1) - \lambda_1$.  Let $C_* =  \mathbb{Z}_{(3)}[\lambda_1, \lambda_2, a]/{a^2  + \lambda_1 a - \lambda_2 a}$. We have the following free resolution of $\pi_*\tmf(2)_{(3)}$ as a $C_*$-module: 
 \[\begin{tikzcd}[column sep = large] 
     \dots \ar{r}{\cdot a} & \Sigma^{8}C_* \ar{rr}{\cdot (a + \lambda_1 - \lambda_2)} &
     & \Sigma^4C_* \ar{r}{\cdot a} & C_* \ar{r} & \pi_*\tmf(2)_{(3)} \ar{r} & 0 
 \end{tikzcd}\]
 Thus, the $E^2$-page of the Tor spectral sequence 
 \[ E^2_{*,*} = \Tor^{C_*}_{*,*}(\pi_*\tmf(2)_{(3)}, \pi_*\tmf(2)_{(3)}) \Longrightarrow
   \pi_*\THH^{\tmf_0(2)_{(3)}}(\tmf(2)_{(3)}) \] 
 is given by 
\[ E^2_{n, *} = \begin{cases}
                        \pi_*\tmf(2)_{(3)},  & n = 0; \\
                        \Sigma^{8k+4}\pi_*\tmf(2)_{(3)}/{(\lambda_1- \lambda_2)}, &
                        n = 2k+1, k \geq 0;  \\
                        0, & \text{ otherwise}. 
                      \end{cases} \]
Since the non-trivial classes in positive filtration  have odd total degree, the  
spectral sequence collapses at the $E^2$-page.  There are no additive extensions, because
the $E^2 = E^{\infty}$-page is a free $\mathbb{Z}_{(3)}$-module. 
\end{proof}
  \subsection{The discriminant map} \label{subsec:discriminant}
If $A \ra B$ is a $G$-Galois extension for a finite group $G$, then
the discriminant map $\mathfrak{d}_{B|A} \colon B \ra F_A(B,A)$ is a
weak equivalence \cite[Proposition 6.4.7]{rognes}. The map
$\mathfrak{d}_{B|A}$ is right adjoint to the \emph{trace pairing} 
\[ \xymatrix@1{B \wedge_A B \ar[r]^(0.6)\mu &  B \ar[r]^{tr} & A} \] 
where $(A \to B) \circ tr$ is homotopic to $\sum_{g \in G} g$ and $\mu$ is the multiplication map of
$B$. Rognes proposes that the deviation of 
$\mathfrak{d}_{B|A}$ from being a weak equivalence might be used for
measuring ramification. We show in the examples of $\ell_p \ra ku_p$
and $ko \ra ku$ that $\mathfrak{d}$ \emph{does} detect ramification, but it
does not give any information about the type of ramification. 

\begin{prop}
There is a cofiber sequence 
\[ \xymatrix@1{
    ku_p \ar[rr]^(0.4){\mathfrak{d}_{ku_p|\ell_p}} & & 
    F_{\ell_p}(ku_p, \ell_p) \ar[r] &
    \bigvee_{i=1}^{p-2} \Sigma^{-2p+2i+2} H\Z_p.   } \]
\end{prop}
\begin{proof}
  We know that $F_{\ell_p}(ku_p, \ell_p)$ can be decomposed as
\[ F_{\ell_p}(ku_p, \ell_p)  \simeq  F_{\ell_p}(\bigvee_{i=0}^{p-2} \Sigma^{2i}
  \ell_p, \ell_p)  \cong \prod_{i=0}^{p-2} \Sigma^{-2i} \ell_p \simeq
  \bigvee_{i=0}^{p-2} \Sigma^{-2i} \ell_p \] 
and $\mathfrak{d}_{ku_p|\ell_p}$ can be identified with a map 
\[ \bigvee_{i=0}^{p-2} \Sigma^{2i} \ell_p \ra   \bigvee_{i=0}^{p-2} \Sigma^{-2i}
  \ell_p. \]                         
As $\pi_*ku_p$ is a free graded $\pi_*\ell_p$-module, we can calculate
the effect of $\mathfrak{d}_{ku_p|\ell_p}$  algebraically via the
trace pairing: The element $\Sigma^{2i}1 \in  \pi_*\Sigma^{2i} \ell_p$
maps an element $u^i$ to $tr(u^{i+j})$ and this is
\[ tr(u^{i+j}) = \begin{cases} 0, & (p-1) \nmid i+j, \\
    (p-1)u^{i+j}, & (p-1) \mid i+j.\end{cases}\]
Hence on the level of homotopy groups $\mathfrak{d}_{ku_p|\ell_p}$
maps $1 \in \pi_0 \Sigma^0 \ell_p$ to $(p-1)\cdot 1 \in \pi_*\Sigma^0\ell_p =
\pi_*\ell_p$ and it maps $\Sigma^{2i}1 \in  \pi_*\Sigma^{2i} \ell_p$
via multiplication with  $(p-1)v_1 = (p-1)u^{p-1}$ to
$\pi_*\Sigma^{-2p+2i+2} \ell_p$. On the summands $\Sigma^{2i}\ell_p$
we get the following maps: 
\[ \xymatrix@C=1cm@R=0.3cm{
\Sigma^{2p-4}\ell_p \ar[rrrrdddd]^{(p-1)v_1}& &  & & \\
\vdots & & &  &\\
\Sigma^2 \ell_p \ar[rrrrdddd]_{(p-1)v_1}& & &  &  \\
\ell_p \ar[rrrr]^{(p-1)}& & &  &\ell_p \\
& & &  &\Sigma^{-2}\ell_p \\
& & &  &\vdots \\
& & &  &\Sigma^{-2p+4} \ell_p}\] 
As $(p-1)$ is a unit in $\pi_0(\ell_p)$ the cofiber of
\[ \xymatrix@1{ \Sigma^{2i} \ell_p \ar[rrr]^{(p-1)v_1}&& & \Sigma^{-2p+2i+2} \ell_p}\]
is $\Sigma^{-2p+2i+2} H\Z_p$. 
\end{proof}

\bigskip
Note that $ko \simeq \tau_{\geq 0}ku^{hC_2}$, but as the trace map $tr
\colon ku \ra ku^{hC_2}$ has the connective spectrum 
$ku$ as a source, it factors through $\tau_{\geq 0}ku^{hC_2}\simeq ko$
and we obtain a discriminant $\mathfrak{d}_{ku|ko} \colon ku \ra
F_{ko}(ku,ko)$. We fix notation for $\pi_*ko$ as
\[ \pi_*ko = \Z[\eta, y, \omega]/(2\eta,\eta^3,\eta y, y^2-4\omega)\]
with $|\eta| = 1$, $|y|=4$ and $|\omega|=8$.

\begin{prop}
There is a cofiber sequence $\xymatrix@1{ku
  \ar[rr]^(0.4){\mathfrak{d}_{ku|ko}} & & F_{ko}(ku, ko) \ar[r]&
  \Sigma^{-2} H\Z}$.  
\end{prop}
\begin{proof}
The cofiber sequence $\xymatrix@1{\Sigma ko \ar[r]^\eta & ko \ar[r]^c &
  ku \ar[r]^\delta&\Sigma^2ko \ar[r]^\eta & \Sigma ko}$ induces a
cofiber sequence 
\[ \xymatrix@1{F_{ko}(\Sigma ko, ko) \ar[r]^{\eta} & F_{ko}(\Sigma^2 ko,
    ko) \ar[r]^{\delta} & F_{ko}(ku, ko) \ar[r]^{c} & F_{ko}(ko,
    ko) \ar[r]^{\eta} & F_{ko}(\Sigma ko, ko)} \]
which is equivalent to
\[ \xymatrix@1{\Sigma^{-1} ko \ar[r]^{\eta} & \Sigma^{-2} ko
    \ar[r]^-{\delta} & F_{ko}(ku, ko) \ar[r]^-{c} & ko \ar[r]^-{\eta}
    & \Sigma^{-1}ko.} \]
This is the $2$-fold desuspension of the cofiber sequence of $ku$ and
hence
\[ F_{ko}(ku, ko) \simeq \Sigma^{-2}ku. \] 

We consider the composition $c_* \circ \mathfrak{d}_{ku|ko} \colon ku \ra
F_{ko}(ku,ko) \ra F_{ko}(ku,ku)$. As $c_*$ is part of the cofiber
sequence 
\[ \xymatrix@1{F_{ko}(ku,\Sigma ko) \ar[r]^{\eta_*} & F_{ko}(ku,ko)
    \ar[r]^{c_*} & F_{ko}(ku,ku) } \] 
and as $\eta$ is trivial on $ku$, we know that $c_*$ induces a monomorphism
on the level of homotopy groups.

As $\mathfrak{d}_{ku|ko}$ is adjoint to the trace pairing, the
composite 
\[ \xymatrix@1{\pi_*ku \ar[r] & \pi_*F_{ko}(ku,ko) \ar[r] &
    \pi_*F_{ko}(ku,ku)}\] can be identified with 
\[ \xymatrix@1{\pi_*ku \ar[r] & \pi_*F_{ko}(ku,ku) \ar[rr]^{(\id +t)_*}
    & & 
    F_{ko}(ku,ku)}\]
where $t$ denotes the generator of $C_2$ and the first map is adjoint
to the multiplication $ku \wedge_{ko} ku \ra ku$.

The target of $c_*$ is $F_{ko}(ku,ku)
\simeq F_{ku}(ku \wedge_{ko} ku, ku)$, and we know by work of the
first author, documented in \cite[Proof of Lemma 0.1]{dlrerratum} that 
\[ \pi_*F_{ku}(ku \wedge_{ko} ku, ku) \cong
  \Hom_{ku_*}(ku_*[s]/(s^2-su), \Sigma^{-*}ku_*), \] so we can control the effect
  of  $c_* \circ \mathfrak{d}_{ku|ko}$ on homotopy groups.

Note that $t$ induces a $ku$-linear map $t_* \colon ku \ra t^*ku$, where $t^*ku$
is the $ku$-module given by restriction of scalars along $t$. 

As $t^2=\id$, we
therefore obtain  
\[ \xymatrix@1{ F_{ko}(ku,ku) \ar[r]^{t_*} & F_{ko}(ku, t^*ku)}\]
and a commutative diagram 
\[ \xymatrix{
F_{ko}(ku, ku) \ar[rr]^t \ar[d]_\simeq & & F_{ko}(ku,ku) \ar[d]^\simeq \\
F_{ku}(ku \wedge_{ko} ku, ku) \ar[rr]^{\beta}& & F_{ku}(ku \wedge_{ko} ku, ku) 
} \]

Here, $\beta$  induces the map on $\pi_*$ that sends an 
$f \colon (ku \wedge_{ko} ku)_* \ra \Sigma^{-i}ku_*$ to 
\[ \xymatrix@1{(ku \wedge_{ko} ku)_* \ar[r]^{(t \wedge \id)_*} & (ku
    \wedge_{ko} ku)_*     \ar[r]^f & \Sigma^{-i}ku_* \ar[r]^t &
    \Sigma^{-i}ku_*}. \]  
If we denote the right unit $\eta_R \colon ku \ra ku \wedge_{ko} ku$
applied to $u$ by $u_r$, then we have the relation $2s +
u_r = u$. As $(t \wedge \id)_*(u) = -u$ and $(t \wedge
\id)_*(u_r) = u_r$, this implies that 
\[ (t \wedge \id)_*(2s) = 2s-2u. \] 
Torsionfreeness then yields $ (t \wedge \id)_*(s) = s-u$. 

The adjoint of the multiplication map $\pi_*ku \ra \pi_*F_{ko}(ku,
ku)$ maps $u^i$ to the map that sends $1$ to $\Sigma^{-2i}u^i$ and $s$
to zero. Therefore, the composite $c_* \circ \mathfrak{d}_{ku|ko}$
maps $u^i$ to the map with values $1 \mapsto
\Sigma^{-2i}(u^i+(-1)^iu^i)$ and
\[   s \mapsto (s-u)u^i \mapsto -t(u^{i+1}) = (-1)^iu^{i+1}. \]
  
In order to understand the effect of $\mathfrak{d}_{ku|ko}$ we
consider the diagram
\[ \xymatrix{\pi_*ku \ar[rrr]^-{(\mathfrak{d}_{ku|ko})_*} & & &
    \pi_*F_{ko}(ku, ko) \cong \pi_{*+2}(ku) \ar[d]_{c^*} \ar[r]^-{c_*}
    & \pi_*F_{ko}(ku,ku) \cong \pi_*(\Sigma^{-2}ku \vee ku)
    \ar[d]_{c^*}\\ 
  &&&\pi_*F_{ko}(ko, ko) \cong \pi_*(ko) \ar[r]^{c_*} &
  \pi_*F_{ko}(ko,ku) \cong \pi_*(ku) }\]
where we can identify $c^* \colon \pi_*F_{ko}(ku, ko) \cong
\pi_{*+2}(ku) \ra \pi_*(ko)$ with $\pi_*\Sigma^{-2}\delta$.

The application of $c^*$ gives the restriction to the unit $c \colon
ko \ra ku$. Say $(\mathfrak{d}_{ku|ko})_*(u^{2}) = x \in
\pi_{6}(ku)$. Then $\pi_*\Sigma^{-2}\delta(x) = \lambda y$,  and as 
$c_*(y)  = 2u^2$, we obtain that
$c^* (\mathfrak{d}_{ku|ko})_*(u^{2}) = \Sigma^{-4}y$ 
and therefore
\[ (\mathfrak{d}_{ku|ko})_*(u^{2}) = u^3. \] 

Similarly $c^*(\mathfrak{d}_{ku|ko})_*(u^{4}) = 
\Sigma^{-8}2\omega$ and $(\mathfrak{d}_{ku|ko})_*(u^{4}) = u^5$ and 
in general  
\[ (\mathfrak{d}_{ku|ko})_*(u^{2i}) =  u^{2i+1}. \]

Restriction to the unit of the odd powers of $u$
gives zero. 

All the $u^i$ send $s$ to $\pm u^{i+1}$ under $c_* \circ
(\mathfrak{d}_{ku|ko})_*$, so also the odd powers of $u$ have to hit a
generator under $(\mathfrak{d}_{ku|ko})_*$, so as a map from $ku$ to
$\Sigma^{-2}ku$ the map  $\mathfrak{d}_{ku|ko}$ has cofiber $\Sigma^{-2}H\Z$. 

\end{proof}

\section{Describing ramification}
\subsection{Log-\'etaleness}

It is shown in \cite{rss} and \cite{sa} that $\ell \ra ku_{(p)}$ is
log-\'etale with respect to the log structures that are generated by
$v_1$ and by $u$. We will use the class $u \in \pi_2ku_{(2)}$
in order to define a pre-log structure for $ko_{(2)} \ra
ku_{(2)}$ 
and show that $ko_{(2)} \ra ku_{(2)}$
is not log-\'etale with respect to this pre-log structure. This indicates
that the map is not tamely ramified. We use the notation from 
\cite{sa}.

Let $\omega$ denote the Bott
element $\omega \in \pi_8ko_{(2)}$. The complexification map sends
$\omega$ to $u^4$. 

By \cite[Lemma 6.2]{sa}  we have an exact sequence
\[ \xymatrix{ &  & & \pi_1\TAQ^C(ku_{(2)})\ar`r[d]`[l]`[llld]`[dll][dll]
    \\
&     \pi_0\Bigl(ku_{(2)} \wedge
\gamma(D(u))/{\gamma(D(w))\Bigr)} \ar[r] &  \pi_0 \TAQ^{(ko_{(2)},
  D(w))}(ku_{(2)}, D(u)) \ar[r] &  \pi_0\TAQ^C(ku_{(2)}),}
\]

where $C = ko_{(2)} \wedge_{S^\J D(w)}S^\J D(u)$ and $D(u)$, $D(\omega)$ are
the pre-log structures for the elements $u$ and $\omega$ as
in \cite[Construction 4.2]{sa}. 
We have that $\gamma(D(w))$ and $\gamma (D(u))$ have the homotopy type
of the sphere and that $\gamma (D(w)) \to \gamma (D(u))$ is
multiplication by $4$. Therefore we get  
\[\pi_0\Bigl(ku_{(2)} \wedge \gamma(D(u))/{\gamma(D(w))\Bigr)} =
  \Z/4\Z.\]    
We want to show that $\pi_1\TAQ^C(ku_{(2)}) = 0 =
\pi_0\TAQ^C(ku_{(2)})$.  By \cite[Lemma 8.2]{b99}  it suffices to show
that $C \to ku_{(2)}$ is an $1$-equivalence. Since $\pi_1(ku_{(2)}) =
0$, it is enough to show that the map is an isomorphism on $\pi_0$.  
Since $S^\J D(w)$ and $S^\J D(u)$ are concentrated in nonnegative
$\J$-space degrees by \cite[Example 6.8]{rss}, they are connective.  
Thus, it is enough to show that $S^\J D(w) \to S^\J D(u)$ induces an
isomorphism on $\pi_0$.  
For this, we only have to prove that $H_0( S^\J D(w), \mathbb{Z}) \to
H_0(S^\J D(u), \mathbb{Z})$ is an isomorphism. Since this map is a
ring map we only need to know that both sides are $\mathbb{Z}$. This
follows from \cite[Proposition 5.2, Corollary 5.3]{rss18}. Hence we obtain
the following result:  
\begin{thm}
  The map 
  $(ko_{(2)}, D(\omega))   \ra (ku_{(2)}, D(u))$ is not log-\'etale. 
\end{thm}

One could try to distinguish between tame and wild ramification by testing
for log-\'etaleness. In many examples, however, it is less obvious what a
suitable log structure would be.


\subsection{Wild ramification and Tate cohomology}
In the algebraic context of Galois extensions of number fields and corresponding
extension of number rings tame 
ramification yields a normal basis and a surjective trace map. Both facts are
actually also sufficient in order to distinguish tame from wild ramification.
For structured ring spectra it does
not work to impose these properties on the level of homotopy groups, because
even for finite faithful Galois extensions these would not hold. Instead we
propose a different criterion that uses the Tate construction. 

\begin{rem}
  Let $G$ be a finite group. 
Usually one calls a $G$-module $M$ \emph{cohomologically trivial}, if
$\hat{H}^i(H;M) =0$, for all $i \in \Z$ and all $H < G$. If $M$ is a
commutative ring $S$, however, it suffices to require $\hat{H}^i(G;S)
=0$ for all $i \in \Z$: In particular, $\hat{H}^0(G;S) =0$, and hence
the trace  map $tr_G \colon S \ra S^G$ is surjective. Thus 
$1_{S^G}$ is in the image of the norm, say $N_G[x] = 1_{S^G}$ for $[x] \in
S_G$. 
If $H < G$ then we consider the diagram
\[ \xymatrix{
    H^0(G;S) = S^G \ar[rr]^{i^*} & & H^0(H;S) = S^H \\
    H_0(G;S) = S_G \ar[u]^{N_G} \ar[rr]^{tr^G_H} && H_0(H;S) = S_H \ar[u]_{N_H}
 }\] 
and therefore we can express can express $1_{S^H}$ as 
\[ 1_{S^H} = i^*(1_{S^G}) = i^*N_G[x] = N_H tr^G_H[x],\]
so $1_{S^H}$ is in the image of $N_H$ and $\hat{H}^0(H;S) =0$. But
$\hat{H}^*(H;S)$ is a graded commutative ring with unit $[1_{S^H}]=0$,
and thus $\hat{H}^*(H;S) = 0$.

The same argument shows that the surjectivity of the trace map suffices
for being cohomologically trivial. 

\end{rem}

Even if $A \ra B$ is a $G$-Galois extension of ring spectra, it is not true,
that this
implies that $B$ is faithful as an $A$-module. An example, due to
Wieland is the $C_2$-Galois extension $F((BC_2)_+, H\F_2) \ra
F((EC_2)_+, H\F_2) \simeq H\F_2$ which is \emph{not} faithful: The
$F((BC_2)_+, H\F_2)$-module spectrum $(H\mathbb{F}_2)^{tC_2}$ is not trivial,
but $H\F_2 \wedge_{F((BC_2)_+, H\F_2)} (H\mathbb{F}_2)^{tC_2} \sim *$.

Note that for a map $A \ra B$ between connective commutative ring spectra with
a finite group $G$ acting on $B$ via commutative $A$-algebra maps
it makes sense to replace the usual homotopy fixed point condition by
the condition that $A$ is weakly
equivalent to $\tau_{\geq 0}B^{hG}$. In many
examples $B^{hG}$ won't be connective. The map $A \ra B$ factors
through $A \ra B^{hG} \ra B$, but as $A$ is connective, we can
consider the induced map on connective covers and obtain a map of
commutative ring spectra 
\[ \tau_{\geq 0}A = A \ra \tau_{\geq 0}B^{hG} \ra \tau_{\geq 0}B =
  B, \]
that turns $\tau_{\geq 0}B^{hG}$ into a commutative $A$-algebra spectrum.

For any spectrum $X$ we denote by $\tau_{<0}X$ the cofiber of the map 
$\tau_{\geq 0} X \ra X$. 

The following result is probably well-known. 
\begin{lem} \label{lem:coftate}
Let $G$ be a finite group and let $e$ be a naive connective $G$-spectrum. Then 
\[ \tau_{\geq 0} e^{hG} \ra  e^{hG} \ra \tau_{<0}e^{tG} \]
is a cofiber sequence and  $\tau_{< 0}e^{tG} \simeq \tau_{< 0}e^{hG}$. 
\end{lem}  
\begin{proof}
  We consider the norm sequence
  \[ \xymatrix@1{e_{hG} \ar[r]^N & e^{hG} \ar[r] & e^{tG}}. \]
  As $e_{hG}$ is a connective spectrum, we have that $\pi_{-1}e_{hG} =
  0$. Hence, applying $\tau_{\geq 0}$ still gives rise to
  a cofiber sequence
 \[ \xymatrix@1{\tau_{\geq 0}e_{hG}=e_{hG} \ar[r]^(0.6){\tau_{\geq
         0}N} & \tau_{\geq 0}e^{hG} 
     \ar[r] & \tau_{\geq 0} e^{tG}}. \]
We combine the norm cofiber sequences with the defining cofiber sequence of
$\tau_{<0}$ and obtain
\[ \xymatrix{ \tau_{\geq 0}e_{hG} \ar@{=}[r] \ar[d]_{\tau_{\geq 0}N}&
    e_{hG} \ar[d]_{N} \ar[r] & \ast \ar[d]\\
    \tau_{\geq 0} e^{hG} \ar[r] \ar[d] & e^{hG} \ar[r] \ar[d] &
    \tau_{<0} e^{hG} \ar[d] \\ 
 \tau_{\geq 0} e^{tG} \ar[r] & e^{tG} \ar[r] & {\tau_{<0} e^{tG}.} 
  } \]
Thus $\tau_{<0} e^{hG} \simeq \tau_{<0} e^{tG}$ and the cofiber
sequence in the second row then yields the claim.  
\end{proof}

\begin{rem}
In many cases, if $B^{tG} \not\simeq *$, then $\pi_*(B^{tG})$ is actually 
periodic. As the canonical K\"unneth map
\[ \pi_*(B^{tG}) \otimes_{\pi_*(B^{hG})} \pi_*(B) \ra \pi_*(B^{tG}
  \wedge_{B^{hG}} B) \]
is  a map of graded commutative rings and as
$\pi_*(B^{tG}) \cong \pi_*(B^{hG})$ in negative degrees, a periodicity
generator in  a negative degree would map to zero in $\pi_*B$ for connective
$B$ and hence  $\pi_*(B^{tG}) \otimes_{\pi_*(B^{hG})} \pi_*(B)$ is the zero ring.
But then also $\pi_*(B^{tG} \wedge_{B^{hG}} B) \cong 0$ and
\[ B^{tG} \wedge_{B^{hG}} B \simeq *. \]
Therefore $B$ would not be a faithful $B^{hG}$-module in these cases. This
emphasizes the importance of replacing the condition that $A$ be weakly
equivalent to $B^{hG}$ by the requirement that $A \simeq \tau_{\geq 0}(B^{hG})$.

From Lemma \ref{lem:coftate} we also know that in order to show that
$B^{tG} \not\simeq *$ for connective $B$ it is 
sufficient to show that $\tau_{< 0}B^{hG}$ is not trivial.
\end{rem}

In \cite[Proposition 6.3.3]{rognes} Rognes assumes that $A \simeq B^{hG}$, but
that assumption is actually not needed for the following:

\begin{thm}
Assume that $G$ is a finite group, $B$ is a cofibrant commutative $A$-algebra
on which $G$ acts via
maps of commutative $A$-algebras. If $B$ is
dualizable and faithful as an $A$-module and if 
\[ h \colon \xymatrix@1{B \wedge_A B \ar[r]^\sim & F(G_+, B)},  \]
then $B^{tG} \simeq *$.  
\end{thm}
The proof is given in \cite{rognes}. We repeat it in order to convince the
reader that one does not need to assume that $A \simeq B^{hG}$.
\begin{proof}
We consider the following commutative diagram in which the columns arise from
the norm cofiber sequence. 
\[
\xymatrix{
B \wedge_A (B_{hG}) \ar[r] \ar[d]_{B \wedge_A N} & (B \wedge_A B)_{hG}
\ar[r]^{h_{hG}} \ar[d]^N &
F(G_+, B)_{hG} \ar[d]^N\\
B \wedge_A (B^{hG}) \ar[r] \ar[d] & (B \wedge_A B)^{hG} \ar[r]^{h^{hG}}
\ar[d] & 
F(G_+, B)^{hG} \ar[d]\\
B \wedge_A (B^{tG}) \ar[r] & (B \wedge_A B)^{tG} \ar[r]^{h^{tG}} & 
F(G_+, B)^{tG}
}
\]

As $F(G_+, B)$ is free over $G$, we have $F(G_+, B)^{tG} \simeq *$ and the 
norm map is an equivalence between $F(G_+, B)_{hG}$ and $F(G_+,B)^{hG}$. 
The map  $h$ is equivariant and hence it induces weak equivalences on 
homotopy orbits, homotopy fixed 
points and the Tate construction. This shows that 
\[ N \colon (B \wedge_A B)_{hG} \ra (B \wedge_A B)^{hG}\]
is a weak equivalence. 

The left horizontal map is a weak equivalence because $B$ is dualizable as 
an $A$-module. The map $B \wedge_A (B_{hG}) \ra (B \wedge_A 
B)_{hG}$ is always a weak equivalence. Therefore the map 
\[ B \wedge_A N \colon B \wedge_A (B_{hG}) \ra B \wedge_A (B^{hG})\]
is a weak equivalence and thus $B \wedge_A (B^{tG}) \simeq *$. As $B$ is a 
faithful $A$-module, this implies that $B^{tG} \simeq *$. 
\end{proof}

\begin{rem}
Therefore, if $G$ is a finite group, $B$ is a cofibrant commutative $A$-algebra
on which $G$ acts via maps of commutative $A$-algebras. If $B$ is
dualizable and faithful as an $A$-module and if 
$B^{tG} \not\simeq *$, then we know that $h \colon B \wedge_A B \ra F(G_+,B)$
cannot be a weak equivalence, and hence $A \ra B$ is ramified. 

\end{rem}

We propose a definition of tame and wild ramification for commutative
ring spectra and justify our definition by investigating several examples.

In the following we denote by $\tau_{\geq 0} X$ the connective cover of a
spectrum $X$. 

\begin{defn}
Assume that $A \ra B$ is a map of commutative ring spectra such that
$G$ acts on $B$ via commutative $A$-algebra maps and $B$ is faithful and
dualizable as an $A$-module. 

If $A \simeq B^{hG}$ (or $A \simeq \tau_{\geq 0}B^{hG}$ if $A$ and
  $B$ are  connective), then we call $A \ra B$ 
  \emph{tamely ramified} if $B^{tG} \simeq *$. Otherwise, $A
  \ra B$ is \emph{wildly ramified}.
\end{defn}

\begin{rem}
To compute the homotopy of $B^{tG}$ we use the Tate spectral sequence 
\[ E^2_{n,m} = \hat{H}^{-n}(G; \pi_m(B)) \Longrightarrow \pi_{n+m}B^{tG}\] 
which is of standard homological type, multiplicative and conditionally convergent. In particular by \cite[Theorem 8.2]{boa}, it converges strongly if it collapses at a finite stage. 
\end{rem}

We will now investigate our criterion for wild ramification in examples.
First, we establish faithfulness: 

\begin{lem} \label{lem:tmf2ff}
The map $\tmf_0(2)_{(3)} \ra \tmf(2)_{(3)}$ identifies $\tmf(2)_{(3)}$
as a faithful $\tmf_0(2)_{(3)}$-module. 
\end{lem}
\begin{proof}
For the map  $\tmf_0(2)_{(3)} \ra \tmf(2)_{(3)}$ we know that $C_2$
acts on $\tmf(2)_{(3)}$ via commutative $\tmf_0(2)_{(3)}$-algebra maps
and that $\tmf_0(2)_{(3)} \simeq \tau_{\geq
  0}(\tmf(2)_{(3)}^{hC_2})$. The trace map $tr \colon \tmf(2)_{(3)}
\ra \tmf(2)_{(3)}^{hC_2}$ factor through $\tau_{\geq
  0}(\tmf(2)_{(3)}^{hC_2}) \simeq \tmf_0(2)_{(3)}$, because
$\tmf(2)_{(3)}$ is connective. As in \cite[Lemma 6.4.3]{rognes} one
can show that the composite  
\[ \xymatrix@1{\tmf_0(2)_{(3)} \simeq \tau_{\geq
      0}(\tmf(2)_{(3)}^{hC_2}) 
\ar[r] &  \tmf(2)_{(3)} \ar[r]^(0.3){tr} & \tau_{\geq
  0}(\tmf(2)_{(3)}^{hC_2})  \simeq  \tmf_0(2)_{(3)} }\]
is homotopic to the map that is the multiplication by $|C_2|=2$. As
$2$ is invertible in $\pi_0\tmf_0(2)_{(3)}$, the trace map $tr \colon
\tmf(2)_{(3)} \ra  \tmf_0(2)_{(3)}$ is a split surjective map of
$\tmf_0(2)_{(3)}$-modules and hence $\tmf_0(2)_{(3)} \ra
\tmf(2)_{(3)}$ is faithful.  
\end{proof}  
  
\begin{lem} \label{lem:tmf02ff} 
The spectrum $\tmf_0(2)_{(3)}$ is faithful as a $\tmf_{(3)}$-module
spectrum. 
\end{lem}  
\begin{proof}
  We already mentioned Behrens' identification \cite[Lemma 2]{behrens}
  $\tmf_0(2)_{(3)}
\simeq \tmf_{(3)} \wedge T$  where $T = S^0 \cup_{\alpha_1} e^4
\cup_{\alpha_1} e^8$ with $\alpha_1 \in (\pi_3S)_{(3)}$. Note that $\alpha_1$ is
nilpotent of order $2$ because $(\pi_6S)_{(3)} = 0$. 

Assume that $M$ is a $\tmf_{(3)}$-module with 
\[ * \simeq M \wedge_{\tmf_{(3)}} \tmf_0(2)_{(3)} \simeq M \wedge
  T. \] 

Then the cofiber sequences 
\[ \xymatrix@1{S^0 \ar[r] & T \ar[r] &  \Sigma^4\cone(\alpha_1)}
  \text{ and }  \xymatrix@1{\cone(\alpha_1) \ar[r] & T \ar[r] & S^8}
\]
imply that $\Sigma^4\cone(\alpha_1) \wedge M  \simeq \Sigma M$ and
$\Sigma^8 M \simeq \Sigma \cone(\alpha_1) \wedge M$ and therefore
\[ \Sigma^{10} M \simeq M. \]

The equivalence is induced by a class  in $\pi_{10}S_{(3)} \cong \mathbb{Z}/{3\mathbb{Z}}\{\beta_1\}$. As this is  nilpotent, we get  that $M \simeq *$.

\end{proof}

\begin{rem}  \label{rem:ffdual}
It is known  that $ko \ra ku$ is faithful \cite[Proposition
5.3.1]{rognes} and dualizable  and it is clear that $\ell \ra ku_{(p)}$ is
faithful and dualizable as
the inclusion of a summand. As $\tmf_1(3)_{(2)}$ can be
identified with $\tmf_{(2)} \wedge DA(1)$ as a $\tmf_{(2)}$-module
\cite[Theorem 4.12]{mathewh},
where $DA(1)$ is a finite cell complex realizing the double of $A(1) = \langle
Sq^1,Sq^2\rangle$, it is dualizable. An argument as in
\cite[Proof of Proposition 5.4.5]{rognes} shows that $\tmf_{(2)} \ra
\tmf_1(3)_{(2)}$ is faithful.

At the moment we don't know whether $\tmf_0(3)_{(2)} \ra
\tmf_1(3)_{(2)}$ is faithful. The diagram
\[ \xymatrix{
\tmf_0(3)_{(2)} \ar[rr] \ar[dr]_{2 \cdot} & & \tmf_1(3)_{(2)}
\ar[ld]^{tr} \\
&\tmf_0(3)_{(2)} & 
  }\] 
commutes, so if $M$ is a $\tmf_0(3)_{(2)}$-module spectrum with $M
\wedge_{\tmf_0(3)_{(2)}} \tmf_1(3)_{(2)} \simeq *$, then
multiplication by $2$ is a trivial self-map on $M$.

Meier shows \cite{meier} 
that $\tmf_1(3)$ is not perfect as a $\tmf_0(3)$-module, hence 
$\tmf_1(3)$ is not a dualizable $\tmf_0(3)$-module.

Meier also proves that $\tmf[\frac{1}{n}] \ra \tmf(n)$ is dualizable for all
$n$. By combining his result with Lemma \ref{lem:tmf2ff} and Lemma
\ref{lem:tmf02ff} we obtain that $\tmf(2)_{(3)}$ is dualizable and
faithful as a $\tmf_{(3)}$-module.

\end{rem}

We show that the extensions $\tmf_0(3)_{(2)} \ra \tmf_1(3)_{(2)}$ and 
$\tmf_{(3)} \ra 
\tmf(2)_{(3)}$ have non-trivial Tate spectra. For $ku$ the Tate spectrum
with respect to the complex conjugation $C_2$-action  satisfies
\[ ku^{tC_2} \simeq \bigvee_{i \in \Z} \Sigma^{4i} H\Z/2\Z. \]
This result is due to Rognes (compare \cite[\S 5.3]{rognes}). 

\begin{thm}
  For $\tmf_1(3)_{(2)}$ with its $C_2$-action we obtain an equivalence of
  spectra 
\[ \tmf_1(3)_{(2)}^{tC_2} \simeq \bigvee_{i \in \Z} \Sigma^{8i} H\Z/2\Z. \] 
\end{thm}
\begin{proof}
 We use the calculations in \cite{mr09}. They compute the homotopy fixed point spectral sequence
 \[ E^2_{n,m} = H^{-n}(C_2; \pi_m\TMF_1(3)_{(2)})
   \Longrightarrow \pi_{n+m}\TMF_0(3)_{(2)}, \] 
 where $\pi_*\TMF_1(3)_{(2)} = \mathbb{Z}_{(2)}[a_1, a_3][\Delta^{-1}]$ 
 with $\Delta = a_3^3(a_1^3-27a_3)$. 
 From their computations we  deduce the following behaviour of the Tate spectral sequence 
\begin{equation} \label{tate1}
  E^2_{n,m} = \hat{H}^{-n}(C_2; \pi_m\TMF_1(3)_{(2)})
  \Longrightarrow \pi_{n+m}\TMF_1(3)_{(2)}^{tC_2}: 
 \end{equation}
Let $R_{n,m}$ be the bigraded ring $\mathbb{Z}/2[a_1, a_3][\Delta^{-1}][\zeta^{\pm}]$ with $|\zeta| = (-1, 0)$. If we assign odd weight to $a_1$, $a_3$ and $\zeta$,   then the $E^2$-page of the Tate spectral sequence is the even part of $R_{n,m}$. 
Alternatively, it is given by 
\[ E^2_{*,*} = S_*[\Delta^{-1}][x^{\pm}], \] 
where
$S_*$ is the subalgebra of $\mathbb{Z}/{2\mathbb{Z}}[a_1,a_3]$ generated by $a_1^2, a_1a_3, a_3^2$, 
 and where $x = \zeta a_3^3 \in E_{-1,18}$. 
 Note that $a_3^2$ is invertible in this ring with $a_3^{-2} = ((a_1a_3)a_1^2-27a_3^2)\Delta^{-1}$.   By Mahowald-Rezk's computations the first non-trivial differential is $d^3$ and we have 
 \begin{align*}
   d^3(a_1^2) = & (x (a_1a_3) a_3^{-4})^3, &  d^3(a_1a_3) = 0,
   &  \qquad  d^3(a_3^2)  = x^3 (a_1a_3) a_3^{-8},  \\
   d^3(x) = & 0, & d^3(\Delta^{-1}) = 0. &    
 \end{align*} 
 Using the Leibniz rule we get that the class $c_{n,m,k,l,i} = (a_1^2)^n(a_1a_3)^m(a_3^2)^k\Delta^{-l}x^i$ with $n,m,k,l \in \mathbb{N}$ and $i \in \mathbb{Z}$ has differential 
\[ d^3(c_{n,m,k,l,i}) = (n+k) x^3 (a_1a_3) a_3^{-10} c_{n,m,k,l,i} . \] 
It follows that $\ker d^3$ is generated as $\mathbb{F}_2$-vector space by the classes $c_{n,m,k,l,i}$ with $n+k = 0$ in $\mathbb{F}_2$.   We claim that 
\[ E^4_{*,*} \cong \mathbb{F}_2[x^{\pm}, \Delta^{\pm}].  \] 
To see this, note the following: 
If $n+k = 0$ in $\mathbb{F}_2$  and $m > 0$, then $c_{n,m,k,l,i}$ is 
zero in $E^4_{*,*}$ because 
\[ d^3(c_{n,m-1,k+5,l,i-3}) = c_{n,m,k,l,i}.\]
 If $n+k = 0$ in $\mathbb{F}_2$ and $n, k > 0$, then we have  $c_{n,0,k,l,i} = c_{n-1, 2, k-1,l,i}$. This  is in the image of $d^3$,  because $n-1 + k-1 = 0$ in $\mathbb{F}_2$ and $ 2 > 0$.  If $n = 0$ in $\mathbb{F}_2$ and $n > 0$, then 
\begin{align*}
 c_{n, 0, 0,l,i}   & =   (a_1^2)^n \Delta^{-l} x^i \\
                       & =  (a_1 a_3)^2 (a_1^2)^{n-1} a_3^{-2} \Delta^{-l} x^i \\
   & = (a_1a_3)^2(a_1^2)^{n-1}((a_1a_3)a_1^2 + a_3^2)\Delta^{-1}\Delta^{-l} x^i \\ 
           &  = c_{n,3,0,l+1,i} + c_{n-1,2,1,l+1,i},         
 \end{align*} 
 and both of these summands   are in the image of $d^3$.  Furthermore, note that in $E^4_{*,*}$ we have
 \[ \Delta = (a_1 a_3)^3  + a_3^4 = c_{0,3,0,0,0} + a_3^4 =  a_3^4.\] 
 This implies that for $k = 0$ in $\mathbb{F}_2$ we have
 \[  c_{0,0,k,l,i}   = (a_3^{4})^{\frac{k}{2}} \Delta^{-l} x^i \equiv \Delta^ {-l + \frac{k}{2}} x^i  \] 
 in $E^4_{*,*}$.  We thus get a surjective map $\mathbb{F}_2[x^{\pm}, \Delta^{\pm}] \to E^4_{*,*}$, which is injective, because the classes $\Delta^lx^i$ for $l,i \in \mathbb{Z}$ are not divisible by $(a_1a_3)$  in $S_*[\Delta^{-1}][x^{\pm}]$. 
 
 From Mahowald-Rezk's computations we get that the next non-trivial  differential is $d^7$ and that we have
 \[ d^7(x) = 0 \text{ and }  d^7(\Delta) = x^7\Delta^{-4} . \] 
 This gives $E^8_{*,*}  = 0$. 
 
 We now want to determine the behaviour of the Tate spectral sequence 
 \begin{equation} \label{tate2}
 E^2_{n,m} = \hat{H}^{-n}(C_2; \pi_m \tmf_1(3)_{(2)}) \Longrightarrow \pi_{n+m} \tmf_1(3)_{(2)}^{tC_2}. 
 \end{equation}
 If we assign again odd weight to $a_1$, $a_3$ and $\zeta$, then the $E^2$-page is the even part of
 \[ \mathbb{Z}/{2\mathbb{Z}}[a_1, a_3][\zeta^{\pm}], \] 
 and one sees that the map of spectral sequences from \eqref{tate2} to \eqref{tate1}  is injective. 
 We get that $d^3$ is the first non-trivial differential in \eqref{tate2}  and that we have
 \begin{align*}
  d^3(a_1a_3)  & = 0,    &  d^3(a_1^2) & = (a_1 \zeta)^3,  \\ 
  d^3(a_3^2 ) & = a_1 a_3^2 \zeta^3,  &  d^3(a_3\zeta) & = (a_1a_3) \zeta^4, \\ 
 d^3(a_1 \zeta) & = 0, &  d^3(\zeta^2) & = a_1 \zeta^5.
 \end{align*}
 Note that an $\mathbb{F}_2$-basis of the $E^3$-page is given by the classes
 \begin{align*}
d_{n,m,i} = &  (a_1^2)^n (a_3^2)^m (\zeta^2)^i, \\
 e_{n,m,i} = & (a_1^2)^n(a_1a_3)(a_3^2)^m (\zeta^2)^i,   \\
f_{n,m,i} = & (a_1^2)^n (a_3^2)^m (a_1 \zeta) (\zeta^2)^i,   \\
g_{n,m,i} = & (a_1^2)^n(a_3^2)^m (a_3 \zeta) (\zeta^2)^i, &  
 \end{align*}
for $n,m \in \mathbb{N}$ and $i \in \mathbb{Z}$.

 The $d^3$-differential on these classes is  given by 
 \begin{align*}
 d^3(d_{n,m,i}) & =   (n+m+i) \cdot f_{n,m,i+1},   \\
 d^3(e_{n,m,i}) & =  (n+m+i) \cdot g_{n+1, m, i+1}, \\ 
 d^3(f_{n,m,i}) & =  (n+m+i) \cdot  d_{n+1,m, i+2}, \\ 
 d^3(g_{n,m,i}) & =  (n+m+i+1) \cdot e_{n,m,i+2}.
 \end{align*} 
 We get  
 \[ E^4_{*,*} =  \bigoplus_{\substack{m \in \mathbb{N}, i \in \mathbb{Z}\\ m+i = 0 \; \text{in} \; \mathbb{F}_2}} \mathbb{F}_2\{d_{0,m,i} \}  \oplus \bigoplus_{\substack{m \in \mathbb{N}, i \in \mathbb{Z} \\ m + i+ 1 = 0 \; \text{in}  \; \mathbb{F}_2} }\mathbb{F}_2\{g_{0, m, i}\}. \] 
 The map of spectral sequences from \eqref{tate2} to \eqref{tate1} 
 satisfies  
 \[ d_{0,m,i} \mapsto \Delta^{\frac{m-3i}{2}}x^{2i}, \qquad  g_{0, m,i} \mapsto \Delta^{\frac{m-3i-1}{2}} x^{2i+1}. \] 
 In particular, one sees that it is injective on $E^4$-pages. 
 We conclude that the next non-trivial differential in spectral sequence \eqref{tate2} is $d^7$ and that we have
 \[ d^7(d_{0,m,i}) = \frac{m-3i}{2} g_{0, m, i+3},  \qquad 
   d^7(g_{0,m,i}) = \frac{m-3i-1}{2} d_{0,m+1,i+4}. \]
 We obtain that
 \[ E^8_{*,*} = \bigoplus_{i \in \mathbb{Z}} \mathbb{F}_2\{d_{0,0,4i}\} =  \bigoplus_{i \in \mathbb{Z}}  \mathbb{F}_2 \{\zeta^{8i}\}. \] 
 Since the $E^8$-page is concentrated in the zeroth row, the spectral sequence collapses at this stage. 
This gives the 
answer on the level of homotopy groups. As $\tmf_1(3)^{tC_2}$ is an
$E_\infty$-ring spectrum \cite{mcc} it is in particular an $E_2$-ring
spectrum and therefore a result by Hopkins-Mahowald (see
\cite[Theorem 4.18]{mnn})
implies that $\tmf_1(3)^{tC_2}$ receives a map from $H\F_2$ and
therefore is a generalized Eilenberg-MacLane spectrum of the claimed
form. 
\end{proof}

\begin{thm}
The $\Sigma_3$-action on $\tmf(2)_{(3)}$ yields
\[ \tmf(2)_{(3)}^{t\Sigma_3} \simeq \bigvee_{i \in \Z}
  \Sigma^{12 i} H\Z/3\Z. \]
\end{thm}
\begin{proof}
We use 
the calculation of \cite{st}. She proves that
$\Tmf(2)_{(3)}^{t\Sigma_3} \simeq *$ via the Tate spectral sequence  
\[ E^2_{n.m}  = \hat{H}^{-n}\bigl(\Sigma_3; \pi_m(\Tmf(2)_{(3)})\bigr)   \Longrightarrow \pi_{n+m}(\Tmf(2)_{(3)}^{t\Sigma_3}).\] 
The $E^2$-page is given by 
\[ \mathbb{Z}/{3\mathbb{Z}}[\alpha, \beta^{\pm}, \Delta^{\pm}]/{\alpha^2}\] 
with $| \alpha | = (-1,4)$, $| \beta | = (-2,12)$ and $| \Delta | = (0,24)$, and the differentials are determined by
\[ d^5(\Delta) = \alpha \beta^2 \quad \text{and} \quad d^9(\alpha \Delta^2) = \beta^5.\] 
Since  $\tmf(2)_{(3)}$ is the connective cover of $\Tmf(2)_{(3)}$ the $E^2$-page of the Tate spectral sequence
\[ \bar{E}^2_{n,m} = \hat{H}^{-n}\bigl(\Sigma_3; \pi_m(\tmf(2)_{(3)})\bigr) \Longrightarrow \pi_{n+m}(\tmf(2)_{(3)}^{t\Sigma_3}) \] 
is the $\mathbb{Z}/{3\mathbb{Z}}$-module 
\[ \bigoplus_{\substack{k,l \in \mathbb{Z} \\ k + 2l \geq 0}} \mathbb{Z}/{3\mathbb{Z}}\{\beta^k \Delta^l\}  \oplus \bigoplus_{\substack{k, l \in \mathbb{Z}  \\ 1 + 3k + 6l \geq 0}} \mathbb{Z}/{3\mathbb{Z}}\{ \alpha \beta^k \Delta^l\}. \] 
Using the map of Tate spectral sequences  $\bar{E}^*_{*,*} \to E^*_{*,*}$   one sees that
\[ \bar{E}^6_{*,*} =  \bigoplus_{\substack{k,l \in \mathbb{Z} \\  k+6l \geq 0}}  \mathbb{Z}/{3 \mathbb{Z}} \{\beta^k (\Delta^{3})^l\}  \oplus \bigoplus_{\substack{k, l \in \mathbb{Z} \\ 13 + 3k + 18l  \geq 0}}  \mathbb{Z}/{3 \mathbb{Z}}\{ (\alpha \Delta^2)\beta^k (\Delta^{3})^l\}.\] 
Since $E^6_{*,*} = \mathbb{Z}/{3\mathbb{Z}}[\alpha \Delta^2, \beta^{\pm}, \Delta^{\pm 3}]/{(\alpha \Delta^2)^2}$  the map $\bar{E}^6_{*,*} \to E^6_{*,*}$ is injective. Thus,  $\bar{d}^9$ is determined by $d^9$  and one gets
\[ \bar{E}^{10}_{*,*}  = \bigoplus_{k \in \mathbb{Z}} \mathbb{Z}/{3\mathbb{Z}}\{(\beta^{-6} \Delta^3)^k\}. \] 
The class $\beta^{-6}\Delta^3$ has bidegree $(12,0)$, and so $\bar{E}^{10}_{*,*}$ is concentrated in line zero and the spectral sequence collapses at this stage. 
  \end{proof}

So we can view the extensions $ko_{(2)} \ra ku_{(2)}$, and $\tmf_{(3)} \ra
\tmf(2)_{(3)}$ as being wildly ramified and
$\tmf_0(3)_{(2)} \ra \tmf_1(3)_{(2)}$ has a non-trivial Tate construction.

In contrast, $KO \ra KU$ is a faithful 
  $C_2$-Galois \cite[\S 5]{rognes}, $\TMF_0(3) \ra \TMF_1(3)$ and
  $\Tmf_0(3) \ra \Tmf_1(3)$ are both faithful $C_2$-Galois extensions
  \cite[Theorem 7.12]{mm15}. In general, $\TMF[1/n] \ra \TMF(n)$ is a faithful
  $GL_2(\Z/n\Z)$-Galois extension \cite[Theorem 7.6]{mm15} and the Tate
  spectrum $\Tmf(n)^{tGL_2(\Z/n\Z)}$ is contractible \cite[Theorem
  7.11]{mm15}. 

\bigskip
For general $n$, constructions of $\tmf_1(n)$ and $\tmf_0(n)$ are tricky: 
For some large $n$, $\pi_1\Tmf_1(n)$ is non-trivial. Lennart Meier constructs
a connective version of $\Tmf_1(n)$ with trivial $\pi_1$ as an $E_\infty$-ring
spectrum  in  \cite{meier}, so that there are $E_\infty$-models of
$\tmf_1(n)$ for all $n$. 

As $\pi_0(\tmf(n)) \cong \Z[\frac{1}{n}, \zeta_n]$ where $\zeta_n$ is
a primitive $n$th root of unity, the defining cofiber sequence of
$\tmf(n)^{tGL_2(\Z/n\Z)}$ gives
\[ \xymatrix@1{{\ldots} \ar[r]& \pi_0\tmf(n)_{hGL_2(\Z/n\Z)} \ar[r]^N &
    \pi_0\tmf(n)^{hGL_2(\Z/n\Z)} \ar[r] & \pi_0
    \tmf(n)^{tGL_2(\Z/n\Z)} \ar[r] & 0}   \]
because $\tmf(n)_{hGL_2(\Z/n\Z)}$ is connective.

By the homotopy orbit spectral sequence we get that
$\pi_0\tmf(n)_{hGL_2(\Z/n\Z)} \cong \Z[\frac{1}{n}, \zeta_n]_{GL_2(\Z/n\Z)}$.
As $\tau_{\geq 0} \tmf(n)^{hGL_2(\Z/n\Z)} \simeq
\tmf[\frac{1}{n}]$, we know that $\Z[\frac{1}{n}] \cong
\pi_0\tmf[\frac{1}{n}]$.

For every $n > 1$ we have  $|GL_2(\Z/n\Z)| = \varphi(n)n^3\prod_{p
  \mid n}(1-\frac{1}{p^2})$. Here, $\varphi$ denotes the Euler
$\varphi$-function and $p$ runs over all primes dividing $n$.

Meier shows \cite{meier}, that $\tmf(n)$ is a perfect $\tmf[1/n]$-module  
spectrum and hence dualizable. In general we do not know whether $\tmf(n)$ is
faithful as a $\tmf[1/n]$-module. For $n=2$, $\tmf(2)_{(3)}$ is faithful and dualizable over $\tmf_{(3)}$, as we saw in Remark \ref{rem:ffdual}.  

We cannot determine the homotopy type of the $GL_2(\Z/n\Z)$-Tate construction
of $\tmf(n)$ for arbitrary $n$, but we can identify cases where it is non-trivial:
\begin{thm}
 Assume $n \geq 2$, $2 \nmid n$ or $n = 2^k$ for some $k \geq 1$. Then 
  \[ \tmf(n)^{tGL_2(\Z/n\Z)} \not\simeq *.\]
\end{thm}
\begin{proof}
  By \cite[p.~282]{km} we know that $GL_2(\Z/n\Z)$ acts on $\zeta_n$ via the
  determinant map: For $A \in GL_2(\Z/n\Z)$ and $\zeta_n$ we get
  \[ A.\zeta_n^i = \zeta_n^{i\cdot \det(A)}. \]
  Therefore, the norm map $N \colon \Z[\frac{1}{n}, \zeta_n]_{GL_2(\Z/n\Z)} \ra
  \Z[\frac{1}{n}]$ sends $\zeta_n^i$ to 
  \[ \sum_{A \in GL_2(\Z/n\Z)} \zeta_n^{i \cdot \det(A)} =
    |SL_2(\Z/n\Z)|\sum_{r \in (\Z/n\Z)^\times} \zeta_n^{ir}, \]
  in particular it sends a primitive $n$-th root of unity $\zeta$ to 
  $|SL_2(\Z/n\Z)|\mu(n)$ with $\mu$ denoting the M\"obius function.

\medskip 
  If $n$ is
  square-free, then $\mu(n) = \pm 1$. Any power of $\zeta_n$ has order $d$ with
  $d \mid n$, so this $d$ is squarefree as well, so $N(\zeta_n^i)$ is equal
  to $|SL_2(\Z/n\Z)|\mu(n)$ or a multiple of it.
  
  \medskip
  If $n$ contains a square of a prime, then $\mu(n) = 0$, so $N(\zeta_n) = 0$
  but of course $N(1) = |GL_2(\Z/n\Z)|$. If $d \mid n$, and $d$ is squarefree,
  then the corresponding power of $\zeta_n$ can give a non-trivial multiple of
  $|SL_2(\Z/n\Z)|$.

  If $2 \nmid n$ then $|GL_2(\Z/n\Z)|$ and $|SL_2(\Z/n\Z)|$ are not units in
  $\Z[\frac{1}{n}]$:
  Let $p$ be an odd prime factor of $n$. Then in $n^3\prod_{p
    \mid n}(1-\frac{1}{p^2})$ we have a factor of $p-1$ and this is even, but $2$ is
  not invertible in $\Z[\frac{1}{n}]$.

  If $n = 2^k$ for some $k \geq 1$, we obtain 
  \[ |SL_2(\Z/n\Z)| = 2^{3k}\frac{3}{4} \]
  which contains $3$ as a non-invertible factor. 
\end{proof}

\begin{rem}
  For many $n$ the Tate construction $\tmf(n)^{tGL_2(\Z/n\Z)}$ is actually trivial. If
  $n = 2^k3^\ell$ with $k, \ell \geq 1$ for instance, the order of $GL_2(\Z/n\Z)$ is  
  invertible in $\Z[\frac{1}{n}]$. Similarly, if $n = 2 \cdot 3 \cdot \ldots \cdot p_m$
  is the product of the first $m$ prime numbers for any $m \geq 2$, then the group order is invertible as well. 
\end{rem}


We close with a periodic example: Let $E_n$ be the Lubin-Tate 
spectrum for the Honda formal group law over $W(\F_{p^n})$. For any finite group $G$, $F(BG_+, E_n) \ra E_n$ is faithful in the $K_n$-local category
\cite[Theorem 4.4]{br}. At the moment we don't know whether $E_n$ is a
dualizable $F(BG_+, E_n)$-module for any finite group $G$. 

In \cite[Theorem 5.1]{br} it is shown that $F((BC_{p^r})_+, E_n) \ra E_n$ is
ramified and one can  also  
consider more general groups than $C_{p^r}$, but the type of ramification was 
\emph{not} determined. The
following result indicates that $F((BC_{p^r})_+, E_n) \ra E_n$ is
wildly ramified. 

\begin{thm} \label{thm:fbgenwild}
  For all $r \geq 1$ and $n \geq 1$
  \[ E_n^{tC_{p^r}} \not\simeq \ast. \] 
\end{thm}
\begin{proof}
  The Tate spectral sequence
\[ E_2^{s,t} = \hat{H}^{-s}(C_{p^r}; \pi_tE_n) \Rightarrow
  \pi_{s+t}(E_n^{tC_{p^r}}) \]
has as $E^2$-term
\[  \hat{H}^{-s}(C_{p^r}; \pi_tE_n) \cong \begin{cases}
    \pi_tE_n^{C_{p^r}}/p^r = \pi_tE_n/p^r, & \text{ for } s \text{ even}, \\
  \ker(N)/\text{im}(t-1) = 0, & \text{ for } s \text{ odd}.\end{cases}   \] 
As $\pi_*(E_n)$ is concentrated in even degrees, the whole $E_2$-term
is concentrated in bidegrees $(s,t)$ where $s$ and $t$ are
even. Therefore, all differentials have to be trivial and $E_2 =
E_\infty$. Thus $\pi_*(E_n^{tC_{p^r}})$ is highly non-trivial.   
\end{proof}

\begin{prop}
Assume that $G$ is a finite group with a non-trivial cyclic subgroup $C_{p^k}
< G$ for some prime $p$. Then $E_n^{tG}$ is non-trivial when $E_n$ is the
Lubin-Tate spectrum at the prime $p$. 
\end{prop}  
\begin{proof}
The restriction map induces a map on Tate constructions $E_n^{tG} \ra
E_n^{t{C_{p^k}}}$. McClure \cite{mcc} shows that the $E_\infty$-structure on Tate 
constructions $E_n^{tG} = t((E_n)_G)^G$ is compatible with inclusions
of subgroups and 
Greenlees-May show \cite[Proposition 3.7]{gm} that for any subgroup $H
< G$ the $H$-spectrum $t((E_n)_G)$ is equivalent to
$t((E_n)_H)$. Therefore the inclusion of fixed points $t((E_n)_G)^G
\ra t((E_n)_G)^H$ is a map of $E_\infty$-ring spectra. As we know that
$E_n^{tC_{p^k}} = t((E_n)_G)^{C_{p^k}}$ is non-trivial by Theorem
\ref{thm:fbgenwild}, $E_n^{tG}$ cannot be trivial, either.  

\end{proof}

\end{document}